\documentclass[a4paper, oneside]{article}

\usepackage{amsmath}
\usepackage{amssymb}
\usepackage{amsthm}
\usepackage{color}
\usepackage{graphicx}
\usepackage{hyperref}
\usepackage{mathrsfs}
\usepackage{pdfsync}
\usepackage{romannum}
\usepackage{theoremref}
\usepackage{titlesec}

\numberwithin{equation}{section}

\titleformat{\section}[hang]
{\bfseries\large}{\thesection.}{0.5em}{}[]
\titlespacing*{\section}{0em}{1em}{1em}

\titleformat{\subsection}[runin]
{\bfseries\normalsize}{\thesubsection.}{0em}{\ }[.]

\theoremstyle{plain}
\newtheorem{thm}{Theorem}[section]
\newtheorem{prop}[thm]{Proposition}
\newtheorem{lem}[thm]{Lemma}
\newtheorem{cor}[thm]{Corollary}

\theoremstyle{definition}

\theoremstyle{remark}
\newtheorem{rem}[thm]{Remark}

\numberwithin{equation}{section}

\DeclareMathOperator{\diag}{diag}		
\DeclareMathOperator{\dom}{dom}		
\DeclareMathOperator{\sgn}{sgn}		
\DeclareMathOperator{\tr}{Tr}			

\def \bE {\mathbb E}		
\def \bN {\mathbb N_+}	
\def \bP {\mathbb P}		
\def \bQ {\mathbb Q}		
\def \bR {\mathbb R}		
\def \bT {\mathbb T}		
\def \bZ {\mathbb Z}		

\def \bff {\mathbf f}		
\def \bfh {\mathbf h}		%
\def \bfp {\mathbf p}		%
\def \bfq {\mathbf q}		%
\def \bfr {\mathbf r}		%
\def \bfu {\mathbf u}		%
\def \bfx {\mathbf x}		
\def \bfy {\mathbf y}		

\def \cA {\mathcal A}		
\def \cL {\mathcal L}		
\def \cS {\mathcal S}		
\def \cX {\mathcal X}		%
\def \cY {\mathcal Y}		%

\def \fJ {\mathfrak J}		

\def \fe {\mathfrak e}		%
\def \fp {\mathfrak p}		%
\def \fr {\mathfrak r}		%
\def \fu {\mathfrak u}		%
\def \fw {\mathfrak w}		%

\def \sC {\mathscr C}		
\def \sF {\mathscr F}		
\def \sG {\mathscr G}		
\def \sH {\mathscr H}		%
\def \sS {\mathscr S}		

\def \bsb {\boldsymbol\beta}		%
\def \bsl {\boldsymbol\lambda}		%
\def \bsm {\boldsymbol\mu}		%
\def \bsn {\boldsymbol\nu}			%
\def \bst {\boldsymbol\tau}			%

\begin{document}

\pagestyle{plain}
\pagenumbering{arabic}
\bibliographystyle{plain}

\title{Equilibrium fluctuation for an anharmonic chain with boundary conditions in the Euler scaling limit}
\author{Stefano \textsc{Olla} \qquad Lu \textsc{Xu}}
\date{}
\maketitle

{\let\thefootnote\relax\footnote{\today}}

\begin{abstract}
  We study the evolution in equilibrium of the fluctuations for the conserved quantities
  of a chain of anharmonic oscillators in the hyperbolic space-time scaling.
  Boundary conditions are determined by applying a constant tension at one side,
  while the position of the other side is kept fixed.
  The Hamiltonian dynamics is perturbed by random terms conservative of such quantities. 
  We prove that these fluctuations evolve macroscopically following
  the linearized Euler equations with the corresponding boundary conditions,
  even in some time scales larger than the hyperbolic one. 
\end{abstract}

\section{Introduction}
\label{sec:introduction}

The deduction of Euler equations for a compressible gas from the microscopic dynamics
under a space-time scaling limit is one of the main problems in statistical mechanics \cite{Morrey55}. 
With a generic assumption of \emph{local equilibrium},
Euler equations can be formally obtained in the limit,
but a mathematical proof starting from deterministic Hamiltonian dynamics is still an open problem. 
The eventual appearance of shock waves complicates further the task,
and in this case, it is expected the convergence to weak entropic solutions of Euler equations. 

Some mathematical results have been obtained by perturbing the Hamiltonian dynamics
by random terms that conserve only energy and momentum,
in such a way that the dynamics has enough ergodicity to generate
some form of local equilibrium (cf. \cite{OVY93, EO14}). 
These results are obtained by relative entropy techniques and
restricted to the smooth regime of the Euler equations. 
The noise introduced in these works are essentially random collisions between
close particles and acts only on the velocities. 
Under such random perturbations, the only conserved quantities
are those that evolve macroscopically with the Euler equations \cite{FFL94}.
Actually, random dynamics and local equilibrium are only tools
to obtain the separation of scales between microscopic and macroscopic
modes necessary in order to close the Euler equations. 
In the deterministic dynamics of harmonic oscillators with random masses (not ergodic),
Anderson localization provides such separation of scales \cite{BHO18}. 

In this article we study the evolution of the fluctuations of the conserved quantities. 
When the system is in equilibrium at certain averaged values of the conserved quantities,
these have Gaussian macroscopic fluctuations. 
The aim is to prove that these fluctuations, in the macroscopic space-time scaling limit,
evolve deterministically following the linearized Euler equations.
It turns out that this is more difficult than proving the hydrodynamic limit,
as it requires the control of the space-time variance of the currents of the conserved quantities. 
More precisely it demands to prove that the currents are \emph{equivalent} (in the norm introduced by the space-time variance) to linear functions of the conserved quantities. 
This step is usually called \emph{Boltzmann-Gibbs principle} (cf. \cite{BR84, KL99}). 
This is the main part of the proof, and it forces us to consider \emph{elliptic} type
of stochastic perturbations, i.e., noise terms that act also on the positions,
not only on the velocities, still maintaining the same conserved quantities. 

The system we consider is $N + 1$ coupled anharmonic oscillators,
similar to the one considered in \cite{EO14}. 
For $i = 0, \ldots, N$, the momentum (or velocity, since we set the masses equal to 1)
of the particle $i$ is denoted by $p_i \in \bR$, while $q_i \in \bR$ denotes its position. 
Particle $0$ is attached to some fixed point, thus $p_0 = 0$, $q_0 = 0$. 
Meanwhile, particle $N$ is pulled (or pushed) by a force $\tau \in \bR$, which is constant in time. 

\vspace{7.5pt}
\begin{figure}[htb]
\begin{center}
\includegraphics[width=0.88\textwidth]{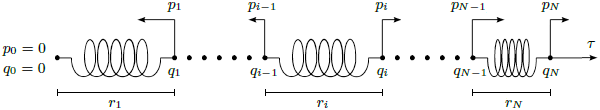}
\end{center}
\end{figure}
\vspace{-7.5pt}

Each pair of consecutive particles $(i-1, i)$ is connected by a (nonlinear) spring with potential $V(q_i - q_{i-1})$. 
We need to assume certain assumptions for the potential energy $V: \bR \to \bR$. 
The energy of the system is then given by 
$$
  H_N(\bfp, \bfq) = \sum_{i=1}^N \left[\frac{p_i^2}{2} + V\big(q_i - q_{i-1}\big)\right]. 
$$
Therefore, the inter-particle distances $\{r_i = q_i - q_{i-1}; 1 \le i \le N\}$ are the essentially relevant variables. 
Notice that here $r_i$ can also assume negative values. 
Let $e_i = p_i^2/2 + V(r_i)$ be the energy assigned to $i$-th particle, then $H_N = \sum e_i$. 
The corresponding Hamiltonian dynamics locally conserves the sums of $p_i$, $r_i$ and $e_i$. 
By adding proper stochastic perturbations on the deterministic dynamics, we can make them the only conserved quantities. 

Let $w_i = (p_i, r_i, e_i)$ be the vector of conserved quantities. 
The hydrodynamic limit is given by the convergence, for any continuous $G$ on $[0, 1]$, 
$$
  \frac 1N \sum_{i=1}^N w_i(Nt)G\left(\frac iN\right)\ \mathop{\longrightarrow}_{N\to\infty}\ \int_0^1 \fw(t, x)G(x)dx, 
$$
where $\fw = (\fp, \fr, \fe)$ solves the \emph{compressible Euler equations} 
\begin{equation}
  \label{eq:compressible euler}
  \partial_t \fw =  \partial_x F(\fw), \quad F(\fw) = \big(\bst(\fr, \fu), \fp, \bst(\fr, \fu)\fp\big), \quad \fu = \fe - \fp^2/2,
\end{equation}
with boundary conditions given by 
$$
  \fp(0, t) = 0, \quad \bst(\fr(1, t), \fu(1, t)) = \tau, 
$$
where $\bst(r, e)$ is the tension function defined in \eqref{eq:temperature tension} later. 
In the smooth regime of \eqref{eq:compressible euler}, this is proven by relative entropy techniques in \cite{EO14}. 

We consider here the system in equilibrium, starting with the Gibbs measure 
\begin{equation}
  \label{eq:gibbs measure}
  \prod_{i=1}^N \exp\left\{\lambda \cdot (r_i, e_i) - \sG(\lambda)\right\}\; dp_i\; dr_i, 
\end{equation}
for given $\lambda = (\beta\tau, -\beta) \in \bR \times \bR_-$, where $\sG$ is the Gibbs potential given by 
\begin{equation}
  \label{eq:gibbs potential}
  \sG(\lambda) = \ln\left(\int_\bR \exp\{-\beta V(r) + \beta\tau r\}dr\right) + \frac12\ln\left(\frac{2\pi}{\beta}\right). 
\end{equation}
Denote by $E_{\lambda,N}$ the expectation with respect to the measure in \eqref{eq:gibbs measure}. 
Correspondingly, there are equilibrium values $0 = E_{\lambda,N} [p_i]$, $\bar r = E_{\lambda,N} [r_i]$, $\bar e = E_{\lambda,N} [e_i]$ for the conserved quantities. 
The empirical distribution of the fluctuations of the conserved quantities is defined by 
$$
  \frac{1}{\sqrt N}\sum_{i=1}^N \begin{pmatrix}p_i(Nt) \\ r_i(Nt) - \bar r \\ e_i(Nt) - \bar e\end{pmatrix}\delta\left(x - \frac iN\right). 
$$
Formally, it is expected to converge to the solution $\tilde\fw = (\tilde\fp, \tilde\fr, \tilde\fe)$ of 
\begin{equation}
  \label{eq:linearized euler}
  \partial_t\tilde\fw = F'(0, \bar r, \bar e)\partial_x \tilde\fw  , 
\end{equation}
where $F'(\bar w)$ is the Jacobian matrix of $F$, with boundary conditions 
\begin{equation}
  \label{eq:linearized bdc}
  \tilde\fp(t, 0) = 0, \quad \frac{\partial \bst}{\partial r}\Big|_{(\bar r, \bar e)}\tilde\fr(t, 1) + \frac{\partial \bst}{\partial e}\Big|_{(\bar r, \bar e)}\tilde\fe(t, 1) = 0, 
\end{equation}
and a proper Gaussian stationary initial distribution. 
Notice that $\tilde \fw(t)$ takes values as distributions on $[0,1]$,
so \eqref{eq:linearized euler} with the boundary conditions \eqref{eq:linearized bdc}
should be intended in the weak sense, as rigorously defined in Section \ref{sec:euler}. 

While the non-equilibrium hydrodynamic limit can be proven by adding a simple exchange of $p_i$ with $p_{i+1}$
at random independent times (cf. \cite{EO14}), in order to prove \eqref{eq:linearized euler}
we need to add, for each bond $(i, i + 1)$, a stochastic perturbation that exchanges
$(p_i, p_{i+1}, r_i, r_{i+1})$
in such way that $r_i + r_{i+1}$, $p_i + p_{i+1}$, $e_i + e_{i+1}$ are conserved. 
The corresponding microcanonical surface is a one-dimensional circle, where we add a Wiener process. 
This stochastic perturbation corresponds to adding a symmetric second order differential operator $\cS_N$
defined by \eqref{eq:generator} that is elliptic on the corresponding microcanonical surfaces. 
The main part of the article is the proof of a lower bound of order $N^{-2}$
on the spectral gap of $\cS_N$ that is independent of the values of the conserved quantities. 
This is an important ingredient for proving the Boltzmann-Gibbs linearization for the dynamics. 


The present article contains the first result on equilibrium fluctuations for anharmonic chain of oscillators with
multiple conserved quantities. Previous results concerned only linear dynamics or vanishing anharmonicity
(eg. \cite{BGJS18} for a system with two conserved quantities).
Another novelty of the present article is the presence of non-linear boundary conditions (tension at the border),
as previous results on equilibrium fluctuations concern systems with no boundary conditions, or linear
in the conserved quantities.

The hyperbolic scale describes the time for the system to reach its mechanical equilibrium. 
Beyond that, it takes more time to reach the thermal equilibrium.
It is a natural question to investigate the behaviour of the equilibrium fluctuations in
larger time scales.
In Theorem \ref{thm:beyond hyperbolic} we prove for our anharmonic system
that the equilibrium fluctuations on the three conserved quantities continue to evolve
deterministically according to the linearized Euler equations up to a time scale $N^a t$ with $a\in [1, 6/5)$.
For harmonic chain with two conserved quantities and no boundary conditions an analogous result
can be found in \cite{BGJSS15}. 
Superdiffusion of energy fluctuations is conjectured in \cite{Spohn14}, and should appear for some $a \ge 3/2$. 
This has been proven rigorously for harmonic chains with conservative noise (cf. \cite{JKO15}
for dynamics with 3 conserved quantities and \cite{BGJ16} with two conserved quantities). 
Results in \cite{JKO15} extends also to the non-stationary superdiffusive evolution of the energy density,
while the other two quantities evolve diffusively \cite{KO16}. See also the review \cite{BBGKO16}
and the other articles in the same volume about the numerical evidence in non-linear dynamics. 
The extension of such superdiffusive results to the non-linear dynamics is one of the most challenging problem. 
Some results for vanishing anharmonicity can be found in \cite{BGJS18}. 

We believe that such macroscopic behavior of the equilibrium fluctuations should be valid also for the deterministic
(non-linear) dynamics, but even the case with a stochastic perturbation acting only on the velocities
remains an open problem.

Another important open problem concerns the evolution of fluctuations out of equilibrium. 
For system with one conserved quantity, like the asymmetric simple exclusion,
in the context of the hyperbolic scaling this has been proven in \cite{Reza95}. 

\section{The microscopic model}
\label{sec:model}

In this section we state the rigorous definition of the microscopic dynamics. 
Let $V$ be a convex, $C^4$-smooth function on $\bR$ with quadratic growth: 
\begin{equation}
  \label{eq:assumptions}
  \inf_{r\in\bR} V''(r) > 0, \quad \sup_{r\in\bR} V''(r) < \infty. 
\end{equation}
Observe that \eqref{eq:assumptions} assures that $V(r)$ acquires its minimum at some unique point $r_0 \in \bR$. 
By replacing $V$ with $V_* = V(\cdot + r_0) - V(r_0)$, we can assume without loss of generality that $V \ge 0$, $V(0) = 0$ and $V'(0) = 0$. 

For $N \ge 1$, let $\Omega_N = \bR^{2N}$ be the configuration space. 
Its elements are denoted by 
$$
  \eta = (\bfp, \bfr); \quad \bfp = (p_1, \ldots, p_N), \ \bfr = (r_1, \ldots, r_N). 
$$
Fix $\tau \in \bR$, $p_0 = 0$, and define first-order differential operators $\cX_i$ acting on smooth functions on $\Omega_N$ by 
$$
  \begin{aligned}
    &\cX_i = (p_i - p_{i-1})\frac{\partial}{\partial r_i} + \big(V'(r_{i+1}) - V'(r_i)\big)\frac{\partial}{\partial p_i}, \quad \text{for }1 \le i \le N - 1, \\
    &\cX_N = (p_N - p_{N-1})\frac{\partial}{\partial r_N} + \big(\tau - V'(r_N)\big)\frac{\partial}{\partial p_N}. 
  \end{aligned}
$$
In addition, define $\cY_{i,i+1}$ for $1 \le i \le N-1$ as 
$$
  \cY_{i,i+1} = (p_{i+1} - p_i)\left(\frac{\partial}{\partial r_{i+1}} - \frac{\partial}{\partial r_i}\right) - \big(V'(r_{i+1}) - V'(r_i)\big)\left(\frac{\partial}{\partial p_{i+1}} - \frac{\partial}{\partial p_i}\right). 
$$
For any $\gamma > 0$, the generator $\cL_N$ is given by 
\begin{equation}
  \label{eq:generator}
  \cL_N = \cA_N + \gamma\cS_N, \quad \cA_N = \sum_{i=1}^N \cX_i, \quad \cS_N = \frac12\sum_{i=1}^{N-1} \cY_{i,i+1}^2. 
\end{equation}
The Liouville operator $\cA_N$ generates the Hamiltonian system introduced in Section \ref{sec:introduction}, while each $\cY_{i,i+1}$ generates a continuous stochastic perturbation on $(p_i, p_{i+1}, r_i, r_{i+1})$, preserving the amounts of $p_i + p_{i+1}$, $r_i + r_{i+1}$ and $e_i + e_{i+1}$. 
This choice of noises assures that $p_i$, $r_i$ and $e_i$ are the only locally conserved quantities. 

Denote by $\pi_{\lambda,N}$ the Gibbs measure in \eqref{eq:gibbs measure}. 
The class of bounded, smooth functions on $\Omega_N$ forms a core of $\cA_N$ and $\cS_N$ in $L^2(\pi_{\lambda,N})$, and for such $f$ and $g$, 
$$
  E_{\lambda,N} \big[(\cA_Nf)g\big] = -E_{\lambda,N} \big[f(\cA_Ng)\big], \quad E_{\lambda,N} \big[(\cS_Nf)g\big] = E_{\lambda,N} \big[f(\cS_Ng)\big]. 
$$
In particular, $\pi_{\lambda,N}$ is stationary with respect to $\cL_N$. 
Moreover, 
$$
E_{\lambda,N} \big[f(-\cL_Nf)\big] = \gamma E_{\lambda,N} \big[f(-\cS_Nf)\big] =
\frac\gamma2\sum_{i=1}^N E_{\lambda,N} \big[(\cY_{i,i+1}f)^2\big]. 
$$
Denote by $\bar w = E_{\lambda,N} [w_i]$, then $\bar w = (0, \bar r(\lambda), \bar e(\lambda))$, where 
\begin{equation}
  \label{eq:canonical center}
  (\bar r(\lambda), \bar e(\lambda)) = \nabla_\lambda\sG(\lambda) = \left(\frac1\beta\frac{\partial}{\partial\tau}, \frac\tau\beta\frac{\partial}{\partial\tau} - \frac{\partial}{\partial\beta}\right)\sG. 
\end{equation}
where $\sG(\lambda)$ is defined in \eqref{eq:gibbs potential}.
It is also worth noticing that the tension in equilibrium is $E_{\lambda,N} [V'(r_i)] = \tau$. 
Furthermore, the covariance matrix $\Sigma = \Sigma(\lambda)$ of $w_i$ under $\pi_{\lambda,N}$ is given by 
\begin{equation}
  \label{eq:covariance}
  \Sigma = E_{\lambda,N} \big[(w_i - \bar w) \otimes (w_i - \bar w)\big] = 
  \left[
  \begin{array}{c|c}
    \beta^{-1} & \begin{matrix} 0 & 0 \end{matrix} \\
    \hline
    \begin{matrix} 0 \\ 0 \end{matrix} & \sG''(\lambda)
  \end{array}
  \right], 
\end{equation}
where $\sG''(\lambda)$ stands for the Hessian matrix of $\sG$ calculated in $\lambda$. 

Define the \emph{thermodynamic entropy} $\sS$ for $r \in \bR$ and $e > 0$ by 
$$
  \sS(r, e) = -\sG^*(r, e), \quad \sG^* = \sup_{\lambda\in\bR\times\bR_-}\big\{ \lambda \cdot (r, e) -\sG(\lambda) \big\}. 
$$
Under our assumptions, $\sG$ is strictly convex and so is its Legendre transform $\sG^*$. 
Hence, $\sS$ is strictly concave. 
By the general theories in Legendre transform, 
\begin{equation}
  \label{eq:convex conjugate}
  \bsl(r, e) = \nabla_{r,e}\sG^*(r, e) = -\nabla_{r,e}\sS(r, e) \in \bR \times \bR_- 
\end{equation}
gives the inverse of $\lambda \to \nabla_\lambda\sG(\lambda)$. 
In view of \eqref{eq:canonical center}, 
\begin{equation}
  \label{eq:inverse}
  \sG''(\lambda) \sS''(\bar r(\lambda), \bar e(\lambda)) = \sG''(\bsl(r, e)) \sS''(r, e) = -I_{2\times2}. 
\end{equation}
For convenience, we denote $\bsl = (\bsb\bst, -\bsb)$, where 
\begin{equation}
  \label{eq:temperature tension}
  \bsb(r, e) = \partial_e\sS(r, e), \quad \bst(r, e) = -\frac{\partial_r\sS(r, e)}{\partial_e\sS(r, e)}. 
\end{equation}
By \eqref{eq:convex conjugate}, $\bsb(r, e)$ is always positive, and 
$$
  \begin{aligned}
    \frac{\partial\bst}{\partial r} + \bst\frac{\partial\bst}{\partial e} &= \frac{1}{\bsb}\left(-\frac{\partial^2\sS}{(\partial r)^2} + \frac{1}{\bsb}\frac{\partial\bsb}{\partial r}\frac{\partial\sS}{\partial r}\right) - \frac{1}{\bsb^2}\frac{\partial\sS}{\partial r}\left(-\frac{\partial^2\sS}{\partial r\partial e} + \frac{1}{\bsb}\frac{\partial\bsb}{\partial e}\frac{\partial\sS}{\partial r} \right) \\
    &= -\frac{1}{\bsb^3}\left(\bsb^2\frac{\partial^2\sS}{(\partial r)^2} - \bsb\frac{\partial^2\sS}{\partial r\partial e}\frac{\partial\sS}{\partial r} - \bsb\frac{\partial^2\sS}{\partial r\partial e}\frac{\partial\sS}{\partial r} + \frac{\partial^2\sS}{\partial e^2}\left[\frac{\partial\sS}{\partial r}\right]^2\right) \\
    &= \frac{1}{\bsb^3}\left(\frac{\partial\sS}{\partial e}, -\frac{\partial\sS}{\partial r}\right) \cdot (-\sS)''\left(\frac{\partial\sS}{\partial e}, -\frac{\partial\sS}{\partial r}\right). 
  \end{aligned}
$$
Since $\sS$ is strictly concave, one can conclude that 
\begin{equation}
  \label{eq:sound speed}
  \frac{\partial\bst}{\partial r} + \bst\frac{\partial\bst}{\partial e} > 0. 
\end{equation}

For each $N \ge 1$, denote by $\{\eta_t \in \Omega_N; t \ge 0\}$ the Markov process generated by $N\cL_N$. 
Observe that $\eta_t = (\bfp(t), \bfr(t))$ can be equivalently expressed by the solution to the following system of stochastic differential equations: 
\begin{equation}
  \label{eq:sde}
  \left\{
  \begin{aligned}
    dp_1(t) &= N\nabla_NV'(r_{1})dt + dJ^p_1,\\
    dp_i(t) &= N\nabla_NV'(r_{i})dt - \nabla_N^*dJ^p_i, \quad \text{for } 2 \le i \le N - 1, \\
    dp_N(t) &= N\big[\tau - V'(r_N)\big]dt  - dJ^p_{N-1}, \\
    dr_1(t) &= Np_1dt + dJ^r_1, \\
    dr_i(t) &= N\nabla_Np_{i-1}dt - \nabla_N^*dJ^r_i, \quad \ \text{for } 2 \le i \le N - 1, \\
    dr_N(t) &= N\nabla_Np_{N-1}dt - dJ^r_{N-1}, 
  \end{aligned}
\right.
\end{equation}
where for any sequence $\{f_i\}$, $\nabla_Nf_i = f_{i+1} - f_i$, $\nabla_N^*f_i = f_{i-1} - f_i$, 
$$
  \begin{aligned}
    dJ^p_i &= \frac{\gamma N}2\big[V''(r_{i+1}) + V''(r_i)\big]\nabla_Np_idt + \sqrt{\gamma N}\big(\nabla_NV'(r_i)\big)dB^i_t,\\
    dJ^r_i &= \gamma N\nabla_NV'(r_i)dt - \sqrt{\gamma N}(\nabla_Np_i)dB^i_t, 
  \end{aligned}
$$
and $\{B^i; i \ge 1\}$ is an infinite system of independent, standard Brownian motions. 
Let $\bP_{\lambda,N}$ be the probability measure on the path space $C([0, \infty), \Omega_N)$ induced by \eqref{eq:sde} and initial condition $\pi_{\lambda,N}$. 
The corresponding expectation is denoted by $\bE_{\lambda,N}$. 

We are interested in the evolution of the fluctuations of the balanced quantities of $\{\eta_t\}$ in macroscopic time. 
For a smooth function $H: [0, 1] \to \bR^3$, define the \emph{empirical distribution of conserved quantities fluctuation field} on $H$ as 
\begin{equation}
  \label{eq:fluctuation}
  Y_N(t, H) = \frac{1}{\sqrt N}\sum_{i=1}^N H\left(\frac iN\right) \cdot \big(w_i(\eta_t) - \bar w \big), \quad \forall t \ge 0, 
\end{equation}
Notice that we consider in \eqref{eq:fluctuation} the \emph{hyperbolic scaling}, where the space and time variables are rescaled by the same order of $N$. 

We close this section with some useful notations. 
Throughout this article, $|\cdot|$ and $\cdot$ always refer to the standard Euclidean norm and inner product in $\bR^d$. 
Let $\sH$ be the space of three-dimensional functions $f = (f_1, f_2, f_3)$ on $[0, 1]$, where each $f_i$ is square integrable. 
The scalar product and norm on $\sH$ are given by 
$$
  \langle f, g \rangle = \int_0^1 f(x) \cdot g(x)dx, \quad \|f\|^2 = \int_0^1 |f(x)|^2dx. 
$$
Then $\sH$ is a Hilbert space, and denote by $\sH'$ its dual space, consisting of all bounded linear functionals on $\sH$. 
Note that the definition in \eqref{eq:fluctuation} satisfies that: 
$$
\bE_{\lambda,N} \left[Y_N^2(t, H)\right] \le |\tr\Sigma(\lambda)| \cdot \frac1N \sum_{i=1}^N \left|H\left(\frac iN\right)\right|^2. 
$$
Thus, one can easily extend the definition of $Y_N(t, H)$ to all $H \in \sH$. 
For all $N \ge 1$, $t \ge 0$ and $H \in \sH$, $Y_N(t, H) \in L^2(\Omega_N; \pi_{\lambda,N})$. 

\section{Euler system with boundary conditions}
\label{sec:euler}

In this section we state the precise definition of the solution to \eqref{eq:linearized euler}, \eqref{eq:linearized bdc} with proper random distribution-valued initial condition. 
The equation \eqref{eq:linearized euler} can be written explicitly as 
$$
  \partial_t\tilde\fp = \tau_r\partial_x\tilde\fr + \tau_e\partial_x\tilde\fe, \quad \partial_t\tilde\fr = \partial_x\tilde\fp, \quad \partial_t\tilde\fe = \tau\partial_x\tilde\fp, 
$$
where $(\tau_r, \tau_e)$ are constants given by 
\begin{equation}
  \label{eq:linear coefficients}
  \tau_r(\lambda) = \frac{\partial}{\partial r}\bst\big(\bar r(\lambda), \bar e(\lambda)\big), \quad \tau_e(\lambda) = \frac{\partial}{\partial e}\bst\big(\bar r(\lambda), \bar e(\lambda)\big). 
\end{equation}
Recall that $(\bsb, \bst)(\bar r(\lambda), \bar e(\lambda)) = (\beta, \tau)$ are constants, and by \eqref{eq:temperature tension}, $\partial_r\sS = -\bsb\bst$, $\partial_e\sS = \bsb$. 
Formally define the linear transformation 
$$
  \tilde\tau = \tau_r\tilde\fr + \tau_e\tilde\fe, \quad \tilde S = -\beta\tau\tilde\fr + \beta\tilde\fe. 
$$
The new coordinates $\tilde\tau$, $\tilde S$ can be viewed as the fluctuation field of tension and thermodynamic entropy, respectively. 
From \eqref{eq:linearized euler}, $(\tilde\fp, \tilde\tau, \tilde S)$ evolves with the equation 
\begin{equation}
  \label{eq:p-system}
  \partial_t\tilde\fp = \partial_x\tilde\tau, \quad \partial_t\tilde\tau = c^2\partial_x\tilde\fp, \quad \partial_t\tilde S = 0, 
\end{equation}
where the constant $c > 0$ is the speed of sound given by 
$$
  c^2 = \tau_r + \tau\tau_e > 0, 
$$
cf. \eqref{eq:sound speed} and \cite[(3.10)]{Spohn14}. 
This transformation also decouples the boundary conditions: 
\begin{equation}
  \label{eq:decoupled bdc}
  \tilde\fp(t, 0) = 0, \quad \tilde\tau(t, 1) = 0. 
\end{equation}
It turns to be clear that $(\tilde\fp, \tilde\tau)$ are two coupled sound modes with mixed boundaries, while $\tilde S$ is independent of $(\tilde\fp, \tilde\tau)$ and does not evolve in time. 
Suppose that the initial data is smooth and satisfies the boundary conditions, one can easily obtain the smooth solution $\tilde\fw = \tilde\fw(t, x)$ to \eqref{eq:linearized euler}, \eqref{eq:linearized bdc} by applying the inverse transformation. 

Since $\tilde\fw(0)$ is a Gaussian random filed, to present the idea above rigorously, we have to consider the weak solution. 
Define a subspace $\sC(\lambda)$ of $\sH$ by 
$$
  \sC(\lambda) = \big\{g=(g_1, g_2, g_3)~\big|~g_i \in C^1([0, 1]),\ g_1(0) = 0,\ \tau_rg_2(1) + \tau_eg_3(1) = 0\big\}. 
$$
Define the first-order differential operator $L$ on $\sC(\lambda)$ by 
$$
  L = B\left(\frac{d}{dx}\right), \quad \text{where } B = F'(\bar w) = 
  \begin{bmatrix}
    0 &\tau_r &\tau_e \\
    1 &0 &0 \\
    \tau &0 &0
  \end{bmatrix}. 
$$
Observe that $B$ has three real eigenvalues $\{0, \pm c\}$, thus generates a hyperbolic system. 
With some abuse of notations, denote the closure of $L$ on $\sH$ still by $L$. 
For $i = 1$, $2$, let $\{\mu_{i,n}; n \ge 0\}$ be two Fourier bases of $L^2([0,1])$ given by 
\begin{equation}
  \label{eq:fourier basis bdc}
  \mu_{1,n}(x) = \sqrt2\sin(\theta_nx), \quad \mu_{2,n}(x) = \sqrt2\cos(\theta_nx), \quad \theta_n = \frac{(2n + 1)\pi}{2}. 
\end{equation}
Notice that $\mu_{1,n}(0) = \mu_{2,n}(1) = 0$, in accordance with the boundary conditions in \eqref{eq:decoupled bdc}. 
For $k \ge 1$, define the Sobolev spaces 
$$
  H_k = \left\{f = (f_1, f_2)~\bigg|~\sum_{i=1}^2\sum_{n=0}^\infty \theta_n^{2k}\left(\int_0^1 f_i(x)\mu_{i,n}(x)dx\right)^2 < \infty\right\}. 
$$
Then $\dom(L) = \{(g_1, \tau_rg_2 + \tau_eg_3) \in H_1\}$. 
To identify the adjoint $L^*$ of $L$, observe that for any $g \in \sC(\lambda)$ and $h \in \sH$, 
$$
  \langle Lg, h \rangle = \int_0^1 g'_1(h_2 + \tau h_3) + (\tau_rg'_2 + \tau_eg'_3)h_1dx. 
$$
Therefore, $\dom(L^*) = \{(h_1, h_2 + \tau h_3) \in H_1\}$. 
In particular, 
$$
  \sC_*(\tau) = \big\{h=(h_1, h_2, h_3)~\big|~h_i \in C^1([0, 1]),\ h_1(0) = 0,\ h_2(1) + \tau h_3(1) = 0\big\} 
$$
is a core of $L^*$ and $L^*h = -B^Th'$ for $h \in \sC_*(\tau)$. 
Notice that $\sC_*(\tau)$ depends only on $\tau$, while $\sC(\lambda)$ depends on both $\beta$ and $\tau$.

Now we can state the definition of \eqref{eq:linearized euler} and \eqref{eq:linearized bdc} precisely. 
Let $\{\tilde\fw(t) = \tilde\fw(t, \cdot); t \ge 0\}$ be a stochastic process taking values in $\sH'$, such that for all $h \in \sC_*(\tau)$, 
\begin{equation}
  \label{eq:weak euler}
  \tilde\fw(t, h) - \tilde\fw(0, h) = \int_0^t \tilde\fw(s, L^*h)ds, \quad \forall t > 0, 
\end{equation}
and $\tilde\fw(0)$ is a Gaussian variable such that for $h$, $g \in \sH$, 
\begin{equation}
  \label{eq:initial distribution}
  E [\tilde\fw(0, h)] = 0, \quad E [\tilde\fw(0, h)\tilde\fw(0, g)] = \langle h, \Sigma g \rangle, 
\end{equation}
where $\Sigma$ is the covariance matrix defined in \eqref{eq:covariance}. 

To see the existence and uniqueness of $\tilde\fw(t)$, consider the weak form of equation \eqref{eq:p-system}: for $f = (f_1, f_2, f_3)$, $f_i \in C^1([0, 1])$, $f_1(0) = f_2(1) = 0$, 
\begin{equation}
  \label{eq:weak wave}
  \tilde\fu(t, f) - \tilde\fu(0, f) + \int_0^t \tilde\fu\left(s, A^Tf'\right)ds = 0, \quad A = 
  \begin{bmatrix}
    0 &1 &0 \\
    c^2 &0 &0 \\ 
    0 &0 &0
  \end{bmatrix}, 
\end{equation}
and $\tilde\fu(0)$ is a centered Gaussian variable with covariance 
$$
  E \big[(\tilde\fu(0, f))^2\big] = \langle f, Qf \rangle, \quad Q = \diag\big(\beta^{-1}, \beta^{-1}c^2, \beta^2\partial_\beta^2\sG\big). 
$$
Suppose that $\{\bsm_{i,n}, \bsn_{i,n}; n \ge 0, i = 1, 2\}$ is the three-dimensional Fourier bases given by $\bsm_{1,n} = (\mu_{1,n}, 0, 0)$, $\bsm_{2,n} = (0, \mu_{2,n}, 0)$, and\footnote[1]{$\bsn_{i,n}$ are chosen arbitrarily, since this coordinate is a constant Gaussian random filed.}
\begin{equation}
  \label{eq:fourier basis no bdc}
  \bsn_{1,n}(x) = \sqrt2\big(0, 0, \sin(\kappa_nx)\big), \quad \bsn_{2,n}(x) = \sqrt2\big(0, 0, \cos(\kappa_nx)\big), \quad \kappa_n = 2n\pi. 
\end{equation}
The solution $\tilde\fu(t)$ is a stationary Gaussian process, satisfying that 
$$
  \begin{aligned}
    &\tilde\fu(t, \bsm_{1,n}) = \frac1{\sqrt\beta}\big(X_{1,n}\cos(c\theta_nt) + X_{2,n}\sin(c\theta_nt)\big), \\
    &\tilde\fu(t, \bsm_{2,n}) = \frac c{\sqrt\beta}\big(X_{1,n}\sin(c\theta_nt) - X_{2,n}\cos(c\theta_nt)\big), \\
    &\tilde\fu(t, \bsn_{i,n}) = \beta\sqrt{\partial_\beta^2\sG(\beta,\tau)} Y_{i,n}, 
  \end{aligned}
$$
where $\{X_{i,n}, Y_{i,n}; n \ge 0, i = 1, 2\}$ is an independent system of standard Gaussian random variables. 
The sample paths $\tilde\fu(\cdot) \in C([0, T]; \sH_{-1})$ a.s., where 
$$
  \sH_{-k} = \left\{\tilde\fu~\bigg|~\sum_{i=1}^2\sum_{n=0}^\infty \big\{\theta_n^{-2k}\tilde\fu^2(\bsm_{i,n}) + \kappa_n^{-2k}\tilde\fu^2(\bsn_{i,n})\big\} < \infty\right\}. 
$$
For each $h \in \sC_*(\tau)$, define $\tilde\fw(t, h)$ by 
\begin{equation}
  \label{eq:rotation}
  \tilde\fw(t, h) = \tilde\fu(t, R^{-1}h), \quad R = R(\lambda) = 
  \begin{bmatrix}
    1 &0 &0 \\ 
    0 &\tau_r &-\beta\tau \\ 
    0 &\tau_e &\beta
  \end{bmatrix}. 
\end{equation}
Observing that $A^TR^{-1} = R^{-1}B^T$, and $f_1(0) = f_2(1) = 0$ for $f \in R^{-1}[\sC_*(\tau)]$, 
$$
  \tilde\fw(t, h) - \tilde\fw(0, h) = -\int_0^t \tilde\fu(s, A^TR^{-1}h')ds = -\int_0^t \tilde\fw(s, B^Th')ds, 
$$
and \eqref{eq:weak euler} is fulfilled. 
On the other hand, from \eqref{eq:temperature tension} and \eqref{eq:linear coefficients}, 
$$
  \begin{pmatrix}\tau_r \\ \tau_e\end{pmatrix} = -\frac1\beta\begin{pmatrix}\partial_r^2\sS \\ \partial_r\partial_e\sS\end{pmatrix}\bigg|_{(\bar r,\bar e)} - \frac\tau\beta\begin{pmatrix}\partial_r\partial_e\sS \\ \partial_e^2\sS\end{pmatrix}\bigg|_{(\bar r,\bar e)}. 
$$
Combining this with \eqref{eq:inverse}, one obtains that 
$$
  \sG''(\lambda)\begin{pmatrix}\tau_r \\ \tau_e\end{pmatrix} = \begin{pmatrix}\beta^{-1} \\ \beta^{-1}\tau\end{pmatrix}. 
$$
By this and some direct calculations, 
$$
  R^T\Sigma R = \diag\big(\beta^{-1}, \beta^{-1}c^2, \beta^2\partial_\beta^2\sG\big) = Q, 
$$
therefore \eqref{eq:initial distribution} also holds. 
In consequence, $\{\tilde\fw(t); t \in [0, T]\}$ uniquely exists in the path space $C([0, T]; \sH_{-k}(\lambda))$ for $k \ge 1$, where 
$$
  \sH_{-k}(\lambda) = \left\{\tilde\fw~\bigg|~\|\tilde\fw\|_{-k}^2 = \sum_{i=1}^2 \sum_{n=0}^\infty \left\{\theta_n^{-2k}\tilde\fw^2\big(R\bsm_{i,n}\big) + \kappa_n^{-2k}\tilde\fw^2\big(R\bsn_{i,n}\big)\right\} < \infty\right\}, 
$$
with three-dimensional Fourier bases $\bsm_{i,n}$ and $\bsn_{i,n}$ given in \eqref{eq:fourier basis bdc} and \eqref{eq:fourier basis no bdc}, and the Gaussian distribution determined by \eqref{eq:initial distribution} is stationary for $\tilde\fw(t)$. 

For $T > 0$ and $k > 5/2$, denote by $\bQ_N$ the distribution of $\{Y_N(t); t \in [0, T]$ on the path space $C([0, T], \sH_{-k}(\lambda))$ induced by $\bP_{\lambda,N}$. 
Denote by $\bQ$ the distribution of $\{\tilde\fw(t); t \in [0, T]\}$ defined above. 
Our first result is stated as below. 

\begin{thm}
\label{thm:hyperbolic}
Assume \eqref{eq:assumptions}, then the sequence of probability measures $\{\bQ_N\}$ converges weakly, as $N \to \infty$,  to the probability measure $\bQ$. 
\end{thm}

\begin{rem}
The condition $k > 5/2$ is necessary only for the tightness in Section \ref{sec:tightness}. 
\end{rem}

Indeed, by the tightness of $\{\bQ_N\}$ in Section \ref{sec:tightness}, we can pick an arbitrary limit point of $\bQ_N$. 
Denote it by $\bQ$ and let $\{Y(t)\}$ be a process subject to $\bQ$. 
From classical central limit theorem, the distribution of $Y(0)$ satisfies \eqref{eq:initial distribution}. 
By virtue of the uniqueness of the solution, to prove Theorem \ref{thm:hyperbolic} it suffices to verify \eqref{eq:weak euler}, or equivalently, 
$$
  \big|Y_N(t, H(t, \cdot)) - Y_N(0, h)\big| \to 0 \quad \text{in probability, }
$$
where $H(t, x)$ solves the \emph{backward Euler system}: 
\begin{equation}
  \label{eq:backward euler}
  \partial_tH(t, x) + L^*H(t, x) = 0, \quad H(0, \cdot) = h, 
\end{equation}
for smooth initial data $h \in \sC_*(\tau)$, with the following additional compatibility conditions also assumed at the space-time edges: 
\begin{equation}
  \label{eq:compatibility conditions}
  \begin{aligned}
    &\lim_{x\to0^+} \partial_xH_1(0, x) = 0, \quad \lim_{x\to1^-} \partial_x\big(H_2(0, x) + \tau H_3(0, x)\big) = 0, \\
    &\lim_{t\to0^+} \partial_t^2H_1(t, 0) = 0, \quad \lim_{t\to0^+} \partial_t^2\big(H_2(t, 1) + \tau H_3(t, 1)\big) = 0. 
  \end{aligned}
\end{equation}
Note that \eqref{eq:compatibility conditions} assures that $H(t, \cdot) \in \sC_*(\tau)$ is differentiable in $x$ up to the second order, and there exists a finite constant $C$ such that 
\begin{equation}
  \label{eq:estimate backward}
  \big|H(t, x)\big| \le C, \quad \big|\partial_xH(t, x)\big| \le C, \quad \big|\partial_x^2H(t, x)\big| \le C 
\end{equation}
for any $t \ge 0$ and $x \in [0, 1]$. 
As a further result of Theorem \ref{thm:hyperbolic}, we are able to prove that the fluctuation field keeps evolving with the linearized system for time scales beyond hyperbolic, under some additional assumptions. 

\begin{thm}
\label{thm:beyond hyperbolic}
Assume \eqref{eq:assumptions}. 
There exists some universal $\delta > 0$, such that if 
\begin{equation}
  \label{eq:additional assumption}
  \sup_{r\in\bR} V''(r) < (1 + \delta)\inf_{r\in\bR} V''(r), 
\end{equation}
then for any $\alpha < 1/5$, $T > 0$ and $\epsilon > 0$, 
\begin{equation}
  \label{eq:beyond hyperbolic}
  \lim_{N\to\infty} \bP_{\lambda,N}\Big\{\exists\hspace{2pt}t \in [0, T], \left|Y_N\big(N^\alpha t, H(N^\alpha t)\big) - Y_N(0, H(0))\right| > \epsilon\Big\} = 0, 
\end{equation}
where $H(t) = H(t, x)$ solves the backward equation \eqref{eq:backward euler} and \eqref{eq:compatibility conditions}. 
\end{thm}

\begin{rem}
Theorem \ref{thm:beyond hyperbolic} shows that the fluctuation of thermodynamic entropy $\tilde S$ keeps stationary for any time scales $N^at$ with $a < 6/5$. 
It is expected that $\tilde S$ would evolve under some superdiffusive scaling $a < 2$ following a fractional heat equation. 
\end{rem}

\begin{rem}
Let $\bT_N = \bZ/(N\bZ)$ be the lattice torus with length $N$. 
One can also put the chain on $\bT_N$ by applying the periodic boundary condition $(p_0, r_0) = (p_N, r_N)$ instead of the ones introduced in Section \ref{sec:introduction}. 
Then, the equilibrium Gibbs measures become 
$$
  \pi_{\beta,\bar p,\tau}^N = \prod_{i\in\bT_N} \exp\left\{\tilde\lambda \cdot w_i - \tilde\sG(\tilde\lambda)\right\}dp_i\;dr_i, 
$$
for given $\tilde\lambda = (\beta \bar p, \beta\tau, -\beta) \in \bR^2 \times \bR_-$, where $\bar p \in \bR$ denotes the momenta in equilibrium. 
For $(p, r) \in \bR^2$ and $e \ge p^2/2 + V(r)$, we can define the internal energy $U = e - p^2/2$, then the thermodynamic entropy and tension function are given by $\sS(r, U)$ and $\bst(r, U)$. 

Start the dynamics from some equilibrium state $\pi_{\beta,\bar p,\tau}^N$. 
Let $\bT = [0, 1)$ stand for the one--dimensional torus. 
For a bounded smooth function $H: \bT \to \bR$, the equilibrium fluctuation field is given by 
$$
  Y_N(t, H) = \frac1{\sqrt N}\sum_{i=1}^N H\left(\frac iN\right) \cdot \begin{pmatrix}p_i(Nt) - \bar p \\ r_i(Nt) - \bar r \\ e_i(Nt) - \bar e\end{pmatrix}. 
$$
With similar argument used to prove Theorem \ref{thm:hyperbolic}, we can show that $Y_N(t, H) \to \tilde\fw(t, H)$. 
Here $\tilde\fw(t, \cdot)$ solves the following linearized Euler system on torus: 
$$
  \partial_t\tilde\fw(t, x) = 
  \begin{bmatrix}
    -\bar p\tau_u &\tau_r &\tau_u \\ 
    1&0 &0 \\ 
    \tau - \bar p^2\tau_u &\bar p\tau_r &\bar p\tau_u 
  \end{bmatrix}
  \partial_x\tilde\fw(t, x), 
$$
where the linear coefficients are given by 
$$
  (\tau_r, \tau_u) = \left(\partial_r, \partial_u \right)\bst\left(\bar r, \bar e - \frac{\bar p^2}2\right). 
$$
Similar to \eqref{eq:p-system}, we have $\tilde\fp$ and $\tilde\tau = -\bar p\tau_u\tilde\fp + \tau_r\tilde\fr + \tau_u\tilde\fe$ form a system of two coupled wave equations with common sound speed $c = \tau_r + \tau\tau_u$, while $\tilde S = \beta(\tilde\fe - \bar p\tilde\fp - \tau\tilde\tau)$ does not evolve in time. 
\end{rem}

\section{Equilibrium fluctuation}
\label{sec:equilibrium fluctuation}

In this section, let $H(t, x)$ be a bounded and smooth function on $[0, \infty) \times [0, 1]$. 
For any $T > 0$, we define two norms $|H|_T$ and $\|H\|_T$ of $H$ as below: 
$$
  |H|_T = \sup_{[0,T]\times[0,1]} |H(t, x)|, 
$$
$$
  \|H\|_T^2 = \sup_{t\in[0,T]} \|H(t)\|^2 = \sup_{t\in[0,T]} \int_0^1 |H(t, x)|^2dx. 
$$
For $Y_N(t, H(t, \cdot))$, the following decomposition holds $\bP_{\lambda,N}$ almost surely: 
\begin{equation}
  \label{eq:ito}
  \begin{aligned}
  &\hspace{14pt}Y_N(t, H(t)) - Y_N(0, H(0)) - \int_0^t Y_N(s, \partial_sH(s))ds \\
  &= I_{N,1}(t, H) + \gamma I_{N,2}(t, H) + \sqrt\gamma M_N(t, H), \quad \forall t > 0, 
  \end{aligned}
\end{equation}
where $I_{N,1}$ and $I_{N,2}$ are integrals given by 
$$
  I_{N,1}(t, H) = N\int_0^t \cA_N[Y_N(s, H(s))]ds, \quad I_{N,2}(t, H) = N\int_0^t \cS_N[Y_N(s, H(s))]ds, 
$$
and $M_N$ is a martingale with quadratic variation given by 
$$
  \langle M_N \rangle(t, H) = N\int_0^t \big\{\cS_N[Y_N^2(s, H(s))] - 2Y_N(s, H(s))\cS_N[Y_N(s, H(s))]\big\}ds. 
$$
As the first step to prove Theorem \ref{thm:beyond hyperbolic}, the next lemma guarantees that the last two terms in \eqref{eq:ito} vanish uniformly in macroscopic time for equilibrium dynamics. 

\begin{lem}
\label{lem:i2 and m}
There exists a constant $C = C(\lambda, V)$, such that 
$$
  \bE_{\lambda,N} \left[\sup_{t\in[0,T]} \gamma\big|I_{N,2}(t,H)\big|^2 + \sup_{t\in[0,T]} \big|M_N(t,H)\big|^2\right] \le \frac{CT}N\|\partial_xH\|_T^2. 
$$
\end{lem}

The proof of Lemma \ref{lem:i2 and m} is standard and we postpone it to the end of this section. 
To identify the boundary conditions of $H$, noting that $p_0 = 0$, and 
$$
  \begin{aligned}
  &\hspace{14pt}N\cA_N[Y_N(t, H(t))] \\
  &= \sqrt N\sum_{i=1}^{N-1} H\left(t, \frac iN\right) \cdot (J_{\cA,i} - J_{\cA,i-1}) + \sqrt NH(t, 1) \cdot \begin{pmatrix}\tau - V(r_N) \\ p_N - p_{N-1} \\ p_N\tau - p_{N-1}V'(r_N)\end{pmatrix} \\
  &= \frac1{\sqrt N} \sum_{i=1}^{N-1} \nabla_{N,i}H(t) \cdot (-J_{\cA,i}(\eta_t)) \\
  &\hspace{11pt}- \sqrt N\left[H_1\left(t, \frac1N\right)\big(V'(r_1(t)) - \tau\big) - \big(H_2(t, 1) + \tau H_3(t, 1)\big)p_N(t)\right], 
  \end{aligned}
$$
where $J_{\cA,i}$ is the centered instantaneous currents of $\cA_N$: 
$$
  J_{\cA,i} = \big(V'(r_{i+1}) - \tau,\ p_i,\ p_iV'(r_{i+1})\big)^T, 
$$
and $\nabla_{N,i}$ is the discrete derivative operator: 
$$
  \nabla_{N,i}H = N\left[H\left(\frac{i+1}{N}\right) - H\left(\frac iN\right)\right]. 
$$
Thus, we can drop the right boundary if $H(t) \in \sC_*(\tau)$ for all $t$: 
\begin{equation}
  \label{eq:i1}
  I_{N,1}(t, H) = -\frac 1{\sqrt N}\int_0^t \sum_{i=0}^{N-1} \nabla_{N,i}H(s) \cdot J_{\cA,i}(\eta_s)ds. 
\end{equation}
The next lemma shows that $I_{N,1}$ can be linearized as $N \to \infty$. 

\begin{lem}
\label{lem:i1}
Assume \eqref{eq:assumptions}, \eqref{eq:additional assumption}, and $H(t) \in \sC_*(\tau)$ for $t \in [0, T]$, then
\begin{equation}
  \label{eq:i1 strong}
  \bE_{\lambda,N} \left[\sup_{t\in[0,T]} \bigg|I_{N,1}(t,H) - \int_0^t Y_N(s, L^*H(s))ds\bigg|^2\right] \le C\left(\frac T{N^\frac15} + \frac{T^2}{N^\frac25}\right) 
\end{equation}
holds with some constant $C$. 
Furthermore, 
$$
  C \le C(\lambda, \gamma, V)|\!|\!|H|\!|\!|_T^2, \quad \text{where} \quad |\!|\!|H|\!|\!|_T^2 = |\partial_xH|_T^2 + |\partial_x^2H|_T^2 + \|\partial_xH\|_T^2. 
$$
\end{lem}

\begin{rem}
\label{rem:i1 weak}
The bound \eqref{eq:i1 strong} in  Lemma \ref{lem:i1} is proven under the assumption \eqref{eq:additional assumption}.
Without assuming \eqref{eq:additional assumption} we have only that, for every fixed $T > 0$, 
\begin{equation}
  \label{eq:i1 weak}
  \bE_{\lambda,N} \left[\sup_{t\in[0,T]} \bigg|I_{N,1}(t,H) - \int_0^t Y_N(s, L^*H(s))ds\bigg|^2\right] \le o_N(1)|\!|\!|H|\!|\!|_T^2. 
\end{equation}
This is clear from Remark \ref{rem:weak boltzmann gibbs} below. 
The bound \eqref{eq:i1 weak} is enough for proving Theorem \ref{thm:hyperbolic}, while \eqref{eq:i1 strong} is necessary in order to prove \eqref{eq:beyond hyperbolic}.
\end{rem}

Lemma \ref{lem:i1} follows from the Boltzmann-Gibbs principle, proven in Section \ref{sec:boltzmann gibbs}.
Here we first give the proof of Theorem \ref{thm:beyond hyperbolic}. 

\begin{proof}[Proof of Theorem \ref{thm:beyond hyperbolic}]
Let $H(t, x)$ be the solution of \eqref{eq:backward euler}. 
From \eqref{eq:ito} and Lemma \ref{lem:i2 and m}, 
$$
  \begin{aligned}
  \bP_{\lambda,N} \bigg\{\exists\hspace{2pt}t \in [0, T],\ &\bigg|Y_N(N^\alpha t, H(N^\alpha t)) - Y_N(0, H(0)) \\
  &- \int_0^{N^\alpha t} Y_N(s, \partial_sH(s))ds - I_{N,1}(N^\alpha t, H)\bigg| > \epsilon\bigg\} \to 0 
  \end{aligned}
$$
for any $\epsilon > 0$. 
Lemma \ref{lem:i1} and \eqref{eq:estimate backward} then yield that for any $\alpha < 1/5$, 
$$
  \bE_{\lambda,N} \left[\sup_{t\in[0,T]} \bigg|I_{N,1}(N^\alpha t, H) - \int_0^{N^\alpha t} Y_N(s, L^*H(s))ds\bigg|^2\right] \to 0. 
$$
Theorem \ref{thm:beyond hyperbolic} then follows from \eqref{eq:backward euler}. 
\end{proof}

For Theorem \ref{thm:hyperbolic}, since tightness is shown in Section \ref{sec:tightness}, we only need to take $\alpha = 0$ in the proof above, and apply Remark \ref{rem:i1 weak} instead of Lemma \ref{lem:i1} in the last step. 

We now proceed to the proof of Lemma \ref{lem:i2 and m}. 
Denote by $\langle \cdot, \cdot \rangle_{\lambda,N}$ the scalar product of two functions $f$, $g \in L^2(\pi_{\lambda,N})$. 
We make use of a well-known estimate on the space-time variance of a stationary Markov process. 
For $f(s, \cdot) \in L^2(\pi_{\lambda,N})$, 
\begin{equation}
  \label{eq:clo estimate}
  \bE_{\lambda,N} \left[\sup_{t \in [0, T]} \bigg|\int_0^t f(s, \eta_s)ds\bigg|^2\right] \le 14\int_0^T \|f(t)\|_{-1,N}^2dt, 
\end{equation}
where $\|f\|_{-1,N}$ is defined for all $f$ on $\Omega_N$ by 
$$
  \|f\|_{-1,N}^2 = \sup_{h} \Big\{2\langle f, h \rangle_{\lambda,N} - \gamma N\langle h, -\cS_Nh \rangle_{\lambda,N}\Big\}, 
$$
with the superior taken over all bounded smooth functions $h$ on $\Omega_N$. 
A proof of \eqref{eq:clo estimate} can be found in \cite[Sec. 2.5]{KLO12}. 

\begin{proof}[Proof of Lemma \ref{lem:i2 and m}]
To begin with, note that 
$$
  \begin{aligned}
    N\cS_N [Y_N(t, H(t))] &= \frac{\sqrt N}2 \sum_{i=1}^{N-1} \left[H\left(t, \frac iN\right) \cdot \cY_{i,i+1}^2[w_i] + H\left(t, \frac{i+1}N\right) \cdot \cY_{i,i+1}^2[w_{i+1}]\right] \\
    &= -\frac 1{\sqrt N}\sum_{i=1}^{N-1} \nabla_{N,i}H(t) \cdot J_{\cS,i}(\eta_t), 
  \end{aligned}
$$
where $J_{\cS,i}$ is the instantaneous current corresponding to $\cS_N$: 
$$
  J_{\cS,i} = \frac12\cY_{i,i+1}^2[w_i] = -\frac12\cY_{i,i+1}^2[w_{i+1}]. 
$$
By applying \eqref{eq:clo estimate} on $I_{N,2}(t, H)$, one obtains that 
$$
  \begin{aligned}
    &\hspace{14pt}\bE_{\lambda,N} \left[\sup_{t \in [0, T]} \big|I_{N,2}(t, H)\big|^2\right] \\
    &\le \frac{14}N\int_0^T \sup_h \left\{2\sum_{i=1}^{N-1} \big\langle \nabla_{N,i}H(t) \cdot J_{\cS,i}, h \big\rangle_{\lambda,N} - \gamma N\big\langle h, -\cS_Nh \big\rangle_{\lambda,N}\right\}dt. 
  \end{aligned}
$$
By Cauchy-Schwarz inequality, with $m_i = \cY_{i,i+1} [w_i]$ for $i = 1$ to $N - 1$, 
$$
  \begin{aligned}
    \left|\sum_{i=1}^{N-1} \big\langle \nabla_{N,i}H(t) \cdot J_{\cS,i}, h \big\rangle_{\lambda,N}\right|^2 &= \frac 14 \left|\sum_{i=1}^{N-1} \big\langle \nabla_{N,i}H(t) \cdot m_i, \cY_{i,i+1}h \big\rangle_{\lambda,N}\right|^2 \\
    &\le \frac 14\sum_{i=1}^{N-1} |\nabla_{N,i}H(t)|^2E_{\lambda,N} \left[|m_i|^2\right] \sum_{i=1}^{N-1} E_{\lambda,N} \left[|\cY_{i,i+1}h|^2 \right] \\
    &\le C_1N \|\partial_xH(t)\|^2\langle h, -\cS_Nh \rangle_{\lambda,N}. 
  \end{aligned}
$$
Substituting this and optimizing $h$, we obtain that 
$$
\bE_{\lambda,N} \left[\sup_{t \in [0, T]} \big|I_{N,2}(t, H)\big|^2\right] \le \frac{14C_1}{\gamma N}\int_0^T \|\partial_xH\|^2dt \le \frac{C_2T}{\gamma N}\|\partial_xH\|_T^2. 
$$
On the other hand, recall that $m_i = \cY_{i,i+1} [w_i]$ and 
$$
  \cS_N \big[Y_N^2(s, H(s))\big] - 2Y_N(s, H(s))\cS_N \big[Y_N(s, H(s))\big] = \frac1{N^3}\sum_{i=1}^{N-1} \big[\nabla_{N,i}H(s) \cdot m_i(\eta_s)\big]^2. 
$$
Therefore, by Doob's maximal inequality, 
$$
  \begin{aligned}
    \bE_{\lambda,N} \left[\sup_{t \in [0, T]} \big|M_N(t, H)\big|^2\right] &\le 4\bE_{\lambda,N} \big[\langle M_N \rangle (T, H)\big] \\
    &\le \frac {4}{N^2}\int_0^T \sum_{i=1}^{N-1} E_{\lambda,N} \big[(\nabla_{N,i}H(t) \cdot m_i)^2\big]dt \\
    &\le \frac{C_3}{N}\int_0^T \|\partial_xH\|^2dt \le \frac{C_4T}{N}\|\partial_xH\|_T^2. 
  \end{aligned}
$$
Since the constants depend only on $\lambda$ and $V$, Lemma \ref{lem:i2 and m} follows. 
\end{proof}

\section{Boltzmann-Gibbs principle}
\label{sec:boltzmann gibbs}

This section devotes to the proof of Lemma \ref{lem:i1}. 
In this section, we denote by $\{\iota_i; 0 \le i \le N\}$ the shift operator semigroup on $\Omega_N$, which is given by 
$$
  (\iota_i\eta)_j = \begin{cases}(p_{i+j}, r_{i+j}), &1 \le j \le N - i, \\ (0, 0), &N - i < j \le N, \end{cases}
$$
for all $\eta \in \Omega_N$ and $0 \le i \le N$. 
For function $F$ on $\Omega_N$, define $\iota_iF = F \circ \iota_i$. 
If $F$ is supported by $\{\eta_j, 1 \le j \le m\}$ for some $m \le N$, then 
$$
  E_{\lambda,N} [\iota_iF] = E_{\lambda,N} [F], \quad \forall 0 \le i \le N - m. 
$$

First notice that $\nabla_{N,i}H$ in \eqref{eq:i1} can be replaced by $\partial_xH$. 
The difference is 
$$
  \begin{aligned}
    &\bE_{\lambda,N} \left[\sup_{t \in [0, T]} \bigg|I_{N,1}(t, H) - \frac 1{\sqrt N}\int_0^t \sum_{i=1}^{N-1} \partial_xH\left(s, \frac iN\right) \cdot J_{\cA,i}(\eta_s)ds\bigg|^2\right] \\
    \le\ &\frac TN\int_0^T E_{\lambda,N} \left[\left(\sum_{i=1}^{N-1} \left[\nabla_{N,i}H(t) - \partial_xH\left(t, \frac iN\right)\right] \cdot J_{\cA,i}\right)^2\right]dt. 
  \end{aligned}
$$
Since $E_{\lambda,N} [J_{\cA,i} \otimes J_{\cA,j}] = \mathbf0$ if $|i - j| > 1$, where $\otimes$ is the tensor product of vectors, the last line in the inequality above has an upper bound 
$$
  \frac TNE_{\lambda,N} \big[|J_{\cA,i}|^2\big]\int_0^T \sum_{i=1}^{N-1} \left|\nabla_{N,i}H(t) - H'\left(t, \frac iN\right)\right|^2dt \le \frac{CT^2|\partial_x^2H|_T^2}{N^2}. 
$$
Clearly the order is better than what is needed for Lemma \ref{lem:i1}. 

Now we want to replace the local random field $J_{\cA,i}$ with its linear approximation. 
The corresponding error can be expressed by 
$$
  \iota_i\Phi = J_{\cA,i} - B(\lambda)\big(w_i - \bar w(\lambda)\big) = \begin{pmatrix}V'(r_{i+1}) - \tau_rr_i - \tau_ee_i \\ 0 \\ p_iV'(r_{i+1}) - p_i\tau\end{pmatrix}. 
$$
Lemma \ref{lem:i1} follows from the following Boltzmann-Gibbs principle. 

\begin{prop}
\label{prop:boltzmann gibbs}
Assume \eqref{eq:assumptions} and \eqref{eq:additional assumption}, then 
$$
  \bE_{\lambda,N} \left[\sup_{t\in[0,T]} \bigg|\frac 1{\sqrt N}\int_0^t \sum_{i=1}^{N-1} \partial_xH\left(s, \frac iN\right) \cdot \iota_i\Phi(\eta_s)ds\bigg|^2\right] \le C\left(\frac{T}{N^\frac15} + \frac{T^2}{N^\frac25}\right) 
$$
for bounded smooth $H = H(t, x)$ on $[0, T] \times [0, 1]$, where $C = C(\lambda, \gamma, V)|\!|\!|H|\!|\!|_T^2$. 
\end{prop}

Boltzmann-Gibbs principle, firstly established for zero range jump process (see \cite{BR84}), aims at determining the space-time fluctuation of a local function by its linear approximation on the conserved fields. 
To show this proposition, we need a spectral gap bound of $\cS_N$, which is the main difficulty here. 
This is established later in Section \ref{sec:spectral gap}. 

\begin{rem}
Notice that the upper bound in Proposition \ref{prop:boltzmann gibbs} is not optimal. 
Indeed, with the proof below, one can actually obtain an upper bound of 
$$
  C\left(\frac T{N^{1-2b}} + \frac{T^2}{N^b}\right), \quad \forall b < \frac12. 
$$
However, this does not improve the time scale in Theorem \ref{thm:beyond hyperbolic}. 
\end{rem}

\begin{proof}
The first step is to take some $1 \le K \ll N$, and define 
$$
  \Phi_K = \frac1K\sum_{i=1}^K \iota_i\Phi. 
$$
We want to replace $\iota_i\Phi$ by $\iota_{i-1}\Phi_K$. 
The error is 
$$
  \sum_{i=1}^{N-1}a_i(t) \cdot \iota_i\Phi - \sum_{i=1}^{N-K} a_i(t) \cdot \iota_{i-1}\Phi_K = F_1(t) + F_2(t), 
$$
where we write $a_i(t) = \partial_xH(t, i/N)$ for short, and $F_1$, $F_2$ are given by 
$$
  \begin{aligned}
    &F_1(t) = \frac1K\left(\sum_{i=1}^{K-1} + \sum_{i=N-K+1}^{N-1}\right) (K-i)a_i(t) \cdot \iota_i\Phi, \\
    &F_2(t) = \frac1K\left(\sum_{i=1}^K \sum_{j=1}^i + \sum_{i=K+1}^{N-K-1} \sum_{j=i+1-K}^i + \sum_{i=N-K}^{N-1} \sum_{j=i+1-K}^{N-K}\right) (a_i(t) - a_j(t)) \cdot \iota_i\Phi. 
  \end{aligned}
$$
Since $E_{\lambda,N} [\iota_i\Phi \otimes \iota_j\Phi] = \mathbf0$ for every pair of $(i, j)$ such that $|i - j| \ge 2$, 
$$
  E_{\lambda,N} \big[F_1^2(t) + F_2^2(t)\big] \le \left(C_1 + \frac{C_2}N\right)K\big(|\partial_xH|_T^2 + |\partial_x^2H|_T^2\big), 
$$
with constants $C_1$ and $C_2$ depending on $\lambda$ and $V$. 
Hence, 
\begin{equation}
  \label{eq:first term}
  \bE_{\lambda,N} \left[\sup_{t\in[0,T]} \bigg|\frac 1{\sqrt N}\int_0^t F_1(s, \eta_s) + F_2(s, \eta_s)ds\bigg|^2\right] \le \frac{C_3T^2K}N\big(|\partial_xH|_T^2 + |\partial_x^2H|_T^2\big). 
\end{equation}

The second step is to replace $\Phi_K$ by its microcanonical center. 
To do so, observe that $\Phi_K$ is supported by $\{\eta_j; 1 \le j \le K + 1\}$, and define 
$$
  \langle \Phi_K \rangle = E_{\lambda,N} \left[\Phi~\bigg|~\frac{w_1 + w_2 + \ldots + w_{K+1}}{K+1}\right], 
$$
where $w_i = (p_i, r_i, e_i)$ is the vector if conserved quantities. 
Due to the equivalence of ensembles (see Section \ref{sec:equivalence of ensembles}), the second moment of $\langle \Phi_K \rangle$ with respect to $\pi_{\lambda,N}$ is of order $K^{-2}$. 
On the other hand, the second moment of $\Phi_K$ is $O(K^{-1})$: 
$$
  E_{\lambda,N} \big[|\Phi_K|^2\big] \le \frac1K\left(E_{\lambda,N}\big[|\iota_1\Phi|^2\big] + 2E_{\lambda,N}\big[\iota_1\Phi \cdot \iota_2\Phi\big]\right), 
$$
Define $\varphi_K = \Phi_K - \langle \Phi_K \rangle$. 
Since $\varphi_K$ and $\langle \Phi_K \rangle$ are orthogonal, 
\begin{equation}
  \label{eq:microcanonical variance bound}
  E_{\lambda,N} \big[|\varphi_K|^2\big] = E_{\lambda,N} \big[|\Phi_K|^2\big] - E_{\lambda,N} [|\langle \Phi_K \rangle|^2\big] \le \frac{C_4}K. 
\end{equation}
By applying the estimate \eqref{eq:clo estimate}, we obtain that 
\begin{equation}
  \label{eq:second term estimate}
  \begin{aligned}
    &\bE_{\lambda,N} \left[\sup_{t \in [0, T]} \bigg|\frac 1{\sqrt N}\int_0^t \sum_{i=1}^{N-K} a_i(s) \cdot \iota_{i-1}\varphi_K(\eta_s)ds\bigg|^2\right] \\
    \le\ &\frac{14}N\int_0^T \sup_h \left\{2\sum_{i=1}^{N-K} \big\langle a_i(t) \cdot \iota_{i-1}\varphi_K, h \big\rangle_{\lambda,N} - \gamma N\langle h, -\cS_Nh \rangle_{\lambda,N}\right\}dt, 
  \end{aligned}
\end{equation}
where the superior is taken over all bounded smooth functions on $\Omega_N$. 
As $\varphi_K$ is supported by $\{\eta_i; 1 \le i \le K + 1\}$, by the spectral gap in Proposition \ref{prop:spectral gap}, 
$$
  -\cS_{K+1}G_{a,K} = a \cdot \varphi_K, \quad a \in \bR^3
$$
can be solved by some function $G_{a,K}$ satisfying that 
$$
\big\langle G_{a,K}, -\cS_{K+1}G_{a,K} \big\rangle_{\lambda,N} \le C(K + 1)^2E_{\lambda,N} \big[(a \cdot \varphi_K)^2\big] \le C_5K|a|^2, 
$$
where the last step follows from \eqref{eq:microcanonical variance bound}. 
For $1 \le i \le N - K$ and $a \in \bR^3$, 
$$
  \big\langle a \cdot \iota_{i-1}\varphi_K, h \big\rangle_{\lambda,N} = \frac12\sum_{j=1}^K \big\langle \cY_{i+j-1,i+j} \big[\iota_{i-1}G_{a,K}\big], \cY_{i+j-1,i+j}h \big\rangle_{\lambda,N}. 
$$
Hence, by Cauchy-Schwarz inequality, 
$$
  \begin{aligned}
    &\left|\sum_{i=1}^{N-K} \big\langle a_i(t) \cdot \iota_{i-1}\varphi_K, h \big\rangle_{\lambda,N}\right|^2 \\
    \le\ &\left(\frac12\sum_{i=1}^{N-K} \sum_{j=1}^K E_{\lambda,N} \big[(\cY_{i+j-1,i+j} h)^2\big]\right)\left(\frac12\sum_{i=1}^{N-K} \sum_{j=1}^K E_{\lambda,N} \big[(\cY_{j,j+1} G_{a_i(t),K})^2\big]\right) \\
    \le\ &K\langle h, -\cS_Nh \rangle_{\lambda,N}\sum_{i=1}^{N-K} \big\langle G_{a_i(t),K}, -\cS_{K+1}G_{a_i(t),K} \big\rangle_{\lambda,N} \\
    \le\ &C_5K^2\langle h, -\cS_Nh \rangle_{\lambda,N}\sum_{i=1}^{N-K} |a_i(t)|^2 \le C_6K^2N\|\partial_xH(t)\|^2\langle h, -\cS_Nh \rangle_{\lambda,N}. 
  \end{aligned}
$$
Substituting this into \eqref{eq:second term estimate} and optimizing in $h$, 
\begin{equation}
  \label{eq:second term}
  \bE_{\lambda,N} \left[\sup_{t \in [0, T]} \bigg|\frac 1{\sqrt N}\int_0^t \sum_{i=1}^{N-K} a_i(s) \cdot \iota_{i-1}\varphi_K(\eta_s)ds\bigg|^2\right] \le \frac{C_7TK^2}{\gamma N}\|\partial_xH\|_T^2. 
\end{equation}

Finally, $\langle \Phi_K \rangle$ is supported by $\{\eta_i; 1 \le i \le K + 1\}$, so that $E_{\lambda,N} [\iota_i\langle \Phi_K \rangle \otimes \iota_j\langle \Phi_K \rangle] = 0$ for $|i - j| \ge K + 2$, and therefore, 
\begin{equation}
  \label{eq:third term}
  \begin{aligned}
    &\bE_{\lambda,N} \left[\sup_{t \in [0, T]} \bigg|\frac 1{\sqrt N}\int_0^t \sum_{i=1}^{N-K} a_i(s) \cdot \iota_{i-1}\langle \Phi_K \rangle(\eta_s)ds\bigg|^2\right] \\
    \le\ &\frac TN\int_0^T \sum_{|i-j|\le K+1} E_{\lambda,N} \left[\big(a_i(t) \cdot \iota_{i-1}\langle \Phi_K \rangle\big)\big(a_j(t) \cdot \iota_{j-1}\langle \Phi_K \rangle\big)\right]dt \\
    \le\ &\frac {T^2\|\partial_xH\|_T^2}N\sum_{i=1}^{N-K} \sum_{j=-K-1}^{K+1} E_{\lambda,N} \big[|\langle \Phi_K \rangle||\iota_j\langle \Phi_K \rangle|\big] \le \frac{C_8T^2}K\|\partial_xH\|_T^2, 
  \end{aligned}
\end{equation}
where the last line is due to that $E_{\lambda,N} [\langle \Phi_K \rangle^2] = O(K^{-2})$. 

In conclusion, by summing up \eqref{eq:first term}, \eqref{eq:second term}, \eqref{eq:third term}, and taking $K = N^{2/5}$, we get the estimate in Proposition \ref{prop:boltzmann gibbs}, with the constant satisfying that 
$$
  C \le C(\lambda, \gamma, V)\big(\|\partial_xH\|_T^2 + |\partial_xH|_T^2 + |\partial_x^2H|_T^2\big). 
$$
This completes the proof of the proposition. 
\end{proof}

\begin{rem}
\label{rem:weak boltzmann gibbs}
If only \eqref{eq:assumptions} is assumed, we can apply Remark \ref{rem:weak spectral gap} instead of Proposition \ref{prop:spectral gap} in the proof of \eqref{eq:second term}. 
By doing this, we can prove Proposition \ref{prop:boltzmann gibbs} for any fixed $T > 0$, with a weaker upper bound $o_N(1)|\!|\!|H|\!|\!|_T^2$. 
\end{rem}

\section{Spectral gap}
\label{sec:spectral gap}

In this section, we state and prove the spectral gap estimate for the dynamics. 
The main result, Proposition \ref{prop:spectral gap}, plays a central role in the proof of Proposition \ref{prop:boltzmann gibbs}. 

Since we want to consider dynamics without boundary conditions in this section, the notations would be slightly different. 
Recall \eqref{eq:assumptions} and denote 
$$
  \delta_- = \inf_{r\in\bR} V''(r), \quad \delta_+ = \sup_{r\in\bR} V''(r). 
$$
For $\beta > 0$, $(\bar p, \tau) \in \bR^2$, let $\pi_{\beta,\bar p,\tau}^K$ be the product measure on $\Omega_K$ given by 
$$
  \pi_{\beta,\bar p,\tau}^K(d\bfp\;d\bfr) = \prod_{i=1}^K \frac1{Z_{\beta,\tau}}\exp\left\{-\frac{\beta(p_i - \bar p)^2}2 - \beta V(r_i) + \beta\tau r_i\right\}dp_i\;dr_i, 
$$
where $Z_{\beta,\tau}$ is the normalization constant. 
Note that the additional coefficient $\bar p$ refers to a nonzero average speed. 
For $K \ge 2$ and $w = (p, r, e)$ such that $e > p^2/2 + V(r)$, the \emph{microcanonical manifold} $\Omega_{w,K}$ is defined as 
$$
  \Omega_{w,K} = \left\{(p_k, r_k), 1 \le k \le K~\left|~\frac{1}{K}\sum_{k=1}^K w_k = w\right.\right\}. 
$$
In view of \eqref{eq:assumptions}, $\Omega_{w,K}$ is a compact and connected manifold. 
The \emph{microcanonical expectation} on $\Omega_{w,K}$ is defined as the conditional expectation 
$$
  E_{w,K} = E_{\pi_{\beta,\bar p,\tau}^K} [\;\cdot\;|\Omega_{w,K}]
$$
Notice that the definition of $E_{w,K}$ is independent of the choice of $\beta$, $\bar p$ or $\tau$. 
For two functions $f_1$, $f_2$ such that $E_{w,K} [f_i^2] < \infty$, we write $\langle f_1, f_2 \rangle_{w,K} = E_{w,K} [f_1f_2]$. 
For each pair $(i, j)$ such that $1 \le i < j \le K$, let $\sF_{i,j}$ be the $\sigma$-algebra over $\Omega_{w,K}$ given by 
$$
  \sF_{i,j} = \sigma (\{(p_k, r_k); 1 \le k \le K, k \not= i, j\}). 
$$

\begin{prop}
\label{prop:spectral gap}
Suppose that the potential $V$ satisfies \eqref{eq:assumptions}. 
There exists a universal constant $\delta > 0$, such that if $V$ fulfills furthermore \eqref{eq:additional assumption}, then 
\begin{equation}
  \label{eq:spectral gap}
  E_{w,K} \left[(f - E_{w,K} [f])^2\right] \le C_K\sum_{k=1}^{K-1} E_{w,K} \left[(\cY_{k,k+1} f)^2\right] 
\end{equation}
for all $(w, K)$ and bounded smooth function $f$, and $C_K \le CK^2$. 
\end{prop}

The proof of Proposition \ref{prop:spectral gap} is divided into Lemma \ref{lem:two points}, \ref{lem:telescopic sum} and \ref{lem:mean field} below. 

\begin{lem}
\label{lem:two points}
Assume \eqref{eq:assumptions}, then there exists constant $C$, such that 
$$
  E_{w,2} \left[(f - E_{w,2} [f])^2\right] \le CE_{w,2} \left[(\cY_{1, 2}f)^2\right] 
$$
for all $w$ and bounded smooth function $f$ on $(p_1, r_1, p_2, r_2)$. 
\end{lem}

\begin{lem}
\label{lem:telescopic sum}
Assume \eqref{eq:assumptions}, then there exists constant $C$, such that 
$$
  \sum_{1 \le i < j \le K} E_{w,K} \left[(f - E_{w,K} [f|\sF_{i,j}])^2\right] \le CK^3\sum_{k=1}^{K-1} E_{w,K} \left[(f - E_{w,K} [f|\sF_{k,k+1}])^2\right] 
$$
for all $K \ge 3$, $w$ and bounded smooth function $f$. 
\end{lem}

\begin{lem}
\label{lem:mean field}
Assume \eqref{eq:assumptions} and \eqref{eq:additional assumption}, then 
\begin{equation}
  \label{eq:mean field}
  E_{w,K} \left[(f - E_{w,K} [f])^2\right] \le C'_K\sum_{1 \le i < j \le K} E_{w,K} \left[(f - E_{w,K} [f|\sF_{i,j}])^2\right] 
\end{equation}
for all $K \ge 3$, $w$ and bounded smooth function $f$, and $C'_K \le C'K^{-1}$. 
\end{lem}

Indeed, for each $k = 1, \ldots, K - 1$, by applying Lemma \ref{lem:two points} to the space $(p_k, r_k, p_{k+1}, r_{k+1})$ and the operator $\cY_{k,k+1}$, one obtains that that 
$$
E_{w,K} \left[(f - E_{w,K} [f|\sF_{k,k+1}])^2 | \sF_{k,k+1}\right] \le C E_{w,K} \left[(\cY_{k,k+1}f)^2 | \sF_{k,k+1}\right]. 
$$
Then, Proposition \ref{prop:spectral gap} turns to be the direct consequence of this, Lemma \ref{lem:telescopic sum} and Lemma \ref{lem:mean field}. 
We now prove these lemmas in turn. 

\begin{proof}[Proof of Lemma \ref{lem:two points}]
For $(p_1, r_1, p_2, r_2) \in \bR^4$, define 
$$
p = p(p_1, p_2) = \frac{p_1 + p_2}{2}, \quad r = r(r_1, r_2) = \frac{r_1 + r_2}{2}, 
$$
and the \emph{internal energy} $E = E(p_1, r_1, p_2, r_2) \ge 0$ given by 
$$
E = \frac{e_1 + e_2}{2} - \frac{p^2}{2} - V(r) = \frac{(p_1 - p_2)^2}{8} + \frac{V(r_1) + V(r_2)}{2} - V\left(\frac{r_1 + r_2}{2}\right). 
$$
Furthermore, let $\theta \in [0, 2\pi)$ satisfy that $\sqrt E\cos\theta = \sqrt 2(p_1 - p_2)/4$ and 
$$
\sqrt E\sin\theta = \sgn(r_1 - r_2)\sqrt{\frac{V(r_1) + V(r_2)}{2} - V\left(\frac{r_1 + r_2}{2}\right)}. 
$$
The Jacobian determinant of the bijection $(p_1, r_1, p_2, r_2) \to (p, r, E, \theta)$ is 
$$
\fJ(p, r, E, \theta) = \sqrt 2 \cdot \frac{\sqrt{V(r_1) + V(r_2) - 2V(r)}}{|V'(r_1) - V'(r_2)|}. 
$$
Recall that $0 < \delta_- \le V''(r) \le \delta_+ < \infty$, we have 
\begin{equation}
  \label{eq:estimate on jacobian}
  0 < \frac{\sqrt{\delta_-}}{\sqrt 2\delta_+} \le \fJ(p, r, E, \theta) \le \frac{\sqrt{\delta_+}}{\sqrt 2\delta_-}. 
\end{equation}
For a bounded smooth function $f = f(p_1, r_1, p_2, r_2)$, define $f_*(p, r, E, \theta) = f(p_1, r_1, p_2, r_2)$, and let $\langle f_* \rangle = \int _0^{2\pi} f_*(p, r, E, \theta)d\theta$. 
By simple calculations, 
$$
E_{w,2} \left[(f - \langle f_* \rangle)^2\right] = \frac{\int_0^{2\pi} [f_*(p, r, E, \theta) - \langle f_* \rangle]^2\fJ(p, r, E, \theta)d\theta}{\int_0^{2\pi} \fJ(p, r, E, \theta)d\theta}. 
$$
On the other hand, since $\cY_{1,2}f = \fJ^{-1}\partial_\theta f_*$, we have 
$$
E_{w,2} \left[(\cY_{1,2}f)^2\right] = \frac{\int_0^{2\pi} [\partial_\theta f_*(p, r, E, \theta)]^2\fJ^{-1}(p, r, E, \theta)d\theta}{\int_0^{2\pi} \fJ(p, r, E, \theta)d\theta}. 
$$
By virtue of the Poincar\'e inequality on one-dimensional torus: 
$$
  \int_0^{2\pi} (f_* - \langle f_* \rangle)^2d\theta \le C\int_0^{2\pi} (\partial_\theta f_*)^2d\theta, 
$$
and the uniform bound of $\fJ$ in \eqref{eq:estimate on jacobian}, we obtain that 
$$
E_{w,2} \left[(f - E_{w,2} [f])^2\right] \le E_{w,2} \left[(f - \langle f_* \rangle)^2\right] \le \frac{C\delta_+}{2\delta_-^2}E_{w,2} \left[(X_{1,2}f)^2\right] 
$$
holds with some universal constant $C < \infty$. 
\end{proof}

\begin{proof}[Proof of Lemma \ref{lem:telescopic sum}]
This lemma is proved along the idea in \cite[Lemma 12.4]{OS13}. 
Below are some notations only used in this proof. 
All of the subscripts $i$, $j$, $k$ are taken form $\{1, \ldots, K\}$. 
We write $x_k = (p_k, r_k)$ and $\bfx = (x_1, \ldots, x_K)$. 
Recall the bijection defined in the proof of the Lemma \ref{lem:two points}. 
For simplicity we write 
$$
(p_{i,j}, r_{i,j}, E(i, j), \theta_{i,j}) = (p, r, E, \theta)(x_i, x_j), \quad \forall i < j. 
$$
For $\theta \in [0, 2\pi)$, denote the Jacobian determinant by 
$$
\fJ_{\bfx,i,j}(\theta) = \fJ\big(p_{i,j}, r_{i,j}, E(i, j), \theta\big). 
$$
For $i < j$, $\theta \in [0, 2\pi]$ and $\bfx = (x_1, \ldots, x_K)$, define a vector $\rho_{i,j}^\theta\bfx$ by 
$$
(\rho_{i,j}^\theta\bfx)_k = \begin{cases}g_1(p_{i,j}, r_{i,j}, E(i, j), \theta), &k = i; \\ g_2(p_{i,j}, r_{i,j}, E(i, j), \theta), &k = j; \\ x_k, &k \not= i, j, \end{cases}
$$
where $(g_1, g_2)$ denotes the inverse map of $(x_1, x_2) \to (p, r, E, \theta)$. 
Observe that $\rho^\theta_{i,j}\bfx = \bfx$ when $\theta = \theta_{i,j}$, and for every smooth function $f$, 
$$
E_{w,K} [f|\sF_{i,j}] = \frac{1}{J_{x_i+x_j}}\int_0^{2\pi} f(\rho_{i,j}^\theta\bfx)\fJ_{\bfx,i,j}(\theta)d\theta, 
$$
where $J_{x_i+x_j} = \int_0^{2\pi} \fJ_{\bfx,i,j}(\theta)d\theta$. 
On the other hand, let $\tau_{i,j}\bfx$ be the vector given by 
$$
(\tau_{i,j}\bfx)_i = x_j, \quad (\tau_{i,j}\bfx)_j = x_i, \quad (\tau_{i,j}\bfx)_k = x_k, \ \forall k \not= i, j. 
$$
Moreover for $1 \le i < j \le K$, we inductively define that 
$$
\sigma_{i,i} = \tilde\sigma_{i,i} = id, \quad \sigma_{i,j} = \tau_{j-1,j} \circ \sigma_{i,j-1}, \quad \tilde \sigma_{i,j} = \tilde\sigma_{i,j-1} \circ \tau_{j-1,j}. 
$$
Observe that for any $i < j$ and $\theta \in [0, 2\pi)$, $ \rho_{i,j}^\theta \equiv \tilde\sigma_{i,j-1} \circ \rho_{j-1,j}^\theta \circ \sigma_{i,j-1}$. 

For a smooth function $f$, by Cauchy-Schwarz inequality, 
$$
\left(f - E_{w,K} [f|\sF_{i,j}]\right)^2 \le \frac{1}{J_{x_i+x_j}}\int_0^{2\pi} \left[f(\rho_{i,j}^\theta\bfx) - f(\bfx)\right]^2\fJ_{\bfx,i,j}(\theta)d\theta. 
$$
The right-hand side is bounded from above by $3(f_1 + f_2 + f_3)$, where 
$$
  \begin{aligned}
    &f_1 = \frac{1}{J_{x_i+x_j}}\int_0^{2\pi} \left[f(\sigma_{i,j-1}\bfx) - f(\bfx)\right]^2\fJ_{\bfx,i,j}(\theta)d\theta, \\
    &f_2 = \frac{1}{J_{x_i+x_j}}\int_0^{2\pi} \left[f(\rho_{j-1,j}^\theta \circ \sigma_{i,j-1}\bfx) - f(\sigma_{i,j-1}\bfx)\right]^2\fJ_{\bfx,i,j}(\theta)d\theta, \\
    &f_3 = \frac{1}{J_{x_i+x_j}}\int_0^{2\pi} \left[f(\tilde\sigma_{i,j-1} \circ \rho_{j-1,j}^\theta \circ \sigma_{i,j-1}\bfx) - f(\rho_{j-1,j}^\theta \circ \sigma_{i,j-1}\bfx)\right]^2\fJ_{\bfx,i,j}(\theta)d\theta. 
  \end{aligned}
$$
For $f_1$, noticing that $f_1 = (f(\sigma_{i,j-1}\bfx) - f(\bfx))^2$, hence 
$$
  \begin{aligned}
    E_{w,K} [f_1] &\le K\sum_{k=i}^{j-2} E_{w,K} \left[(f \circ \sigma_{i,k+1} - f \circ \sigma_{i,k})^2\right] \\
    &= K\sum_{k=i}^{j-2} E_{w,K} \left[(f \circ \tau_{k,k+1} - f)^2\right]. 
  \end{aligned}
$$
Notice that $E_{w,K} [f \circ \tau_{k,k+1} | \sF_{k,k+1}] = E_{w,K} [f|\sF_{k,k+1}]$, so that 
$$
  E_{w,K} \left[(f \circ \tau_{k,k+1} - E_{w,K}[f|\sF_{k,k+1}])^2\right] = E_{w,K} \left[(f - E_{w,K} [f|\sF_{k,k+1}])^2\right]. 
$$
This together with the convex inequality $(a + b)^2 \le 2(a^2 + b^2)$ yields that 
$$
  E_{w,K} [f_1] \le 4K\sum_{k=i}^{j-2} E_{w,K} \left[(f - E_{w,K}[f|\sF_{k,k+1}])^2\right]. 
$$
For $f_2$, by applying the change of variable $\bfy = \sigma_{i,j-1}\bfx$, we obtain that 
$$
  E_{w,K} [f_2] = E_{w,K} \left[\frac{1}{J_{y_{j-1}+y_j}}\int_0^{2\pi} \left[f(\rho_{j-1,j}^\theta\bfy) - f(\bfy)\right]^2\fJ_{\bfy,i,j}(\theta)d\theta\right]. 
$$
Therefore, we can calculate this term as 
$$
  \begin{aligned}
    E_{w,K} [f_2] &= 2E_{w,K} [f^2] -2E_{w,K} [fE_{w,K} [f|\sF_{j-1,j}]] \\
    &= E_{w,K} \left[(f - E_{w,K} [f|\sF_{j-1,j}])^2\right]. 
  \end{aligned}
$$
For $f_3$, the same change of variable yields that 
$$
  E_{w,K} [f_3] = E_{w,K} \left[E_{w,K} [(f \circ \tilde\sigma_{i,j-1} - f)^2~|~\sF_{j-1,j}]\right] = E_{w,K} \left[(f \circ \tilde\sigma_{i,j-1} - f)^2\right]. 
$$
Since $\tilde\sigma_{k,j-1} = \tau_{k,k+1} \circ \tilde\sigma_{k+1,j-1}$, by repeating the calculation in $f_1$, 
$$
  E_{w,K} [f_3] \le 4K\sum_{k=i}^{j-2} E_{w,K} \left[(f - E_{w,K} [f|\sF_{k,k+1}])^2\right]. 
$$
Hence, with some universal constant $C < \infty$ we have 
$$
  E_{w,K} \left[(f - E_{w,K} [f|\sF_{i,j}])^2\right] \le CK\sum_{k=i}^{j-1} E_{w,K} \left[(f \circ \tau_{k,k+1} - f)^2\right]. 
$$
Lemma \ref{lem:telescopic sum} follows by summing up this estimate with $i$ and $j$. 
\end{proof}

To show Lemma \ref{lem:mean field}, we need the following pre-estimate. 

\begin{lem}
\label{lem:weak mean field}
Assume \eqref{eq:assumptions}, then \eqref{eq:mean field} holds with constants $C'_K$ satisfying 
$$
  C'_K \le \frac{C'}K\left(\frac{\delta_+}{\delta_-}\right)^{3(K-1)}. 
$$
\end{lem}

\begin{rem}
\label{rem:weak spectral gap}
In view of Lemma \ref{lem:weak mean field}, the spectral gap in \eqref{eq:spectral gap} also holds without the assumption \eqref{eq:additional assumption}. 
In this case, the constants $C_K$ satisfies that 
$$
  C_K \le CK^2\left(\frac{\delta_+}{\delta_-}\right)^{3(K-1)}. 
$$
\end{rem}

We first prove Lemma \ref{lem:mean field} from Lemma \ref{lem:weak mean field}. 
The proof of Lemma \ref{lem:weak mean field} is put in the end of this section. 
Consider the bounded operator 
$$
  \cL_Kf = \frac{1}{K}\sum_{1 \le i < j \le K} \left(E_{w,K} [f|\sF_{i,j}] - f\right), \quad \forall f\ \text{s.t.}\ E_{w,K} [f^2] < \infty. 
$$
Let $\lambda_{w,K}$ be the spectral gap of $\cL_K$ on with respect to $E_{w,K}$: 
$$
  \lambda_{w,K} \triangleq \inf \left\{\langle f, -\cL_Kf \rangle_{w,K}~|~E_{w,K} [f] = 0, E_{w,K} [f^2] = 1\right\}, 
$$
and let $\lambda_K = \inf \{\lambda_{w,K}; w \in \bR^2 \times \bR_+\}$. 
Then \eqref{eq:mean field} is equivalent to 
$$
\inf \{\lambda_K; K \ge 3\} > 0. 
$$
We prove it through an induction argument, firstly established for $k = 3$, $4$ in \cite{Caputo08}. 

\begin{lem}
\label{lem:induction}
If $k\lambda_k \ge 1$ holds for some $k \ge 3$, then for all $K \ge k$, 
$$
  \lambda_K \ge (k\lambda_k - 1)\left(\frac{1}{k - 2} - \frac{2}{K(k - 2)}\right) + \frac 1K. 
$$
\end{lem}

In view of \eqref{eq:additional assumption} and Lemma \ref{lem:weak mean field}, for some fixed $k$ which is large enough, 
$$
  k\lambda_k > \frac k{C'}\left(\frac{\delta_-}{\delta_+}\right)^{3k - 3} \ge \frac k{C'}\frac{1}{(1 + \delta)^{3k-3}} \ge 1, 
$$
provided that $\delta > 0$ is small enough. 
Then, with Lemma \ref{lem:induction} we can show that the sequence $\{\lambda_K; K \ge 3\}$ is uniformly bounded from below. 

\begin{proof}[Proof of Lemma \ref{lem:induction}]
We make use of the equivalent characterization of $\lambda_{w,K}$ that 
$$
  \lambda_{w,K} = \inf \left\{\frac{\langle \cL_Kf, \cL_Kf \rangle_{w,K}}{\langle f, -\cL_Kf \rangle_{w,K}}~\Big|~\langle f, -\cL_Kf \rangle_{w,K} \not= 0\right\}. 
$$
In this proof we denote by $B$ the set of all pairs $b = (i, j)$ such that $1 \le i < j \le K$,  and write $D_bf = E_{w,K} [f|\sF_b] - f$ for all $b \in B$, then 
$$
  \langle \cL_Kf, \cL_Kf \rangle_{w,K} = \frac{1}{K^2}\sum_{b, b' \in B} \langle D_bf, D_{b'}f \rangle_{w,K}, 
$$
$$
  \langle f, -\cL_Kf \rangle_{w,K} = \frac 1K \sum_{b \in B} \langle D_bf, D_bf \rangle_{w,K}. 
$$
We write $b \sim b'$ if two pairs $b$ and $b'$ have at least one common point. 
We also consider all the $k$-particle subsets $T_k \subseteq \{1, \ldots, K\}$. 
Notice that if $b \sim b'$ but $b \not= b'$, there are $\binom {K - 3}{k - 3}$ different $T_k$'s containing both $b$ and $b'$. 
Hence, 
$$
  \binom{n - 3}{k - 3}\sum_{\substack{b, b' \in B \\ b \not= b',b \sim b'}} \langle D_bf, D_{b'}f \rangle_{w,K} = \sum_{T_k} \sum_{\substack{b, b' \subseteq T_k \\ b \not= b', b \sim b'}} \langle D_bf, D_{b'}f \rangle_{w,K}. 
$$
If $b \not\sim b'$, there are $\binom{K - 4}{k - 4}$ different $T_k$'s contain both $b$ and $b'$, while for the case $b = b'$ it is $\binom{K - 2}{k - 2}$. 
Therefore, the right-hand side of the equation above equals to 
$$
  \sum_{T_k} \sum_{b, b' \subseteq T_k} \langle D_bf, D_{b'}f \rangle_{w,K} - \binom{K - 4}{k - 4}\sum_{b \not\sim b'} \langle D_bf, D_{b'}f \rangle_{w,K} - \binom{K - 2}{k - 2}\sum_{b \in B} \langle D_bf, D_bf \rangle_{w,K}. 
$$
The definition of $\lambda_k$ yields that 
$$
  \frac{1}{k}\sum_{b, b' \subseteq T_k} \langle D_bf, D_{b'}f \rangle_{w,K} \ge \lambda_k\sum_{b \subseteq T_k} \langle D_bf, D_bf \rangle_{w,K}. 
$$
And for $b \not\sim b'$, $\langle D_bf, D_{b'}f \rangle_{w,K} = E_{w,K} \left[(D_{b'}D_bf)^2\right] \ge 0$. 
Therefore, 
$$
  \sum_{\substack{b, b' \in B \\ b \not= b',b \sim b'}} \langle D_bf, D_{b'}f \rangle_{w,K} \ge \frac{(k\lambda_k - 1)(K - 2)}{k - 2}\sum_{b \in B} \langle D_bf, D_bf \rangle_{w,K}. 
$$
By the condition $k\lambda_k > 1$, the right-hand side is positive. 
In conclusion, 
$$
  \begin{aligned}
    \langle \cL_Kf, \cL_Kf \rangle_{w,K} &\ge \frac{1}{K^2}\sum_{b \in B} \langle D_bf, D_bf \rangle_{w,K} + \frac{1}{K^2}\sum_{b \not= b', b \sim b'} \langle D_bf, D_bf \rangle_{w,K} \\
    &\ge \frac{1}{K^2}\left[\frac{(k\lambda_k - 1)(K - 2)}{k - 2} + 1\right]\sum_{b \in B} \langle D_bf, D_bf \rangle_{w,K} \\
    &= \left[(k\lambda_k - 1)\left(\frac{1}{k - 2} - \frac{2}{K(k - 2)}\right) + \frac 1K\right]\langle f, - \cL_Kf \rangle_{w,K}. 
  \end{aligned}
$$
Notice that this estimate is independent of the choice of $w$. 
\end{proof}

Finally, to complete the proof of Proposition \ref{prop:spectral gap}, we are left to show Lemma \ref{lem:weak mean field}. 
To do this, we make use of the spectral gap bound of Kac walk. 
For $a \in \bR^2$ and $R \ge |a|^2$, consider the $(2K - 3)$-dimensional sphere 
$$
S_K(a,R) = \left\{x_1, \ldots, x_K \in \bR^2~\bigg|~\frac{1}{K}\sum_{k=1}^K x_k = a, \ \frac{1}{K}\sum_{k=1}^K |x_k|^2 = R\right\}. 
$$
Denote by $\mu_K(a, R)$ the uniform measure on $S_K(a, R)$. 
With a little abuse of notations, let $\sF_{i,j} = \sigma\{x_k; k \not=i, j\}$ for $1 \le i < j \le K$. 

\begin{lem}
\label{lem:kac walk}
There exists a constant $C$ such that 
$$
  E_{\mu_K(a,R)} \left[(f - E_{\mu_K(a,R)} [f])^2\right] \le \frac{C}{K}\sum_{1 \le i < j \le n} E_{\mu_K(a,R)} \left[(f - E_{\mu_K(a,R)} [f | \sF_{i,j}])^2\right] 
$$
for all $(a, R, K)$ and bounded smooth function $f$. 
\end{lem}

Lemma \ref{lem:kac walk} can be proved by the arguments in \cite{CCL03} and \cite{CGL08}. 
We here prove Lemma \ref{lem:weak mean field} by applying a perturbation on the spectral gap in Lemma \ref{lem:kac walk}. 

\begin{proof}
To begin with, from \eqref{eq:assumptions} we know that for $r \not= r'$ and $K \ge 1$, 
\begin{equation}
  \label{eq:estimate on jacobian-k}
  \frac{\sqrt{2(K + 1)}}{\sqrt K}c_- \le \frac{|V'(r) - V'(r')|}{\sqrt{V(r) + KV(r') - (K + 1)V\big(\frac{r + Kr'}{K + 1}\big)}} \le \frac{\sqrt{2(K + 1)}}{\sqrt K}c_+, 
\end{equation}
where $c_- = \delta_-/\sqrt{\delta_+}$ and $c_+ = \delta_+/\sqrt{\delta_-}$. 
For each $K \ge 3$, we construct a bijection $\tau_K: \Omega_K \to \Omega_K$, satisfying the following two conditions. 
\begin{itemize}
  \item[(\romannum1)] For $w = (p, r, e)$, $\tau_K(\Omega_{w,K}) = S_K(a, R)$, where $a = (p, r)$, $R = 2e - 2V(r) + r^2$; 
  \item[(\romannum2)] The Jacobian matrix $\tau'_K$ of $\tau_K$ satisfies that $c_-^{K-1} \le |\det(\tau'_K)| \le c_+^{K-1}$. 
\end{itemize}
Indeed, given a bounded, measurable, positive function $g$ on $\Omega_{w,K}$, by (\romannum1) we know that $\tau_K^{-1}g := g \circ \tau_K^{-1}$ defines a function on $S_K(a, R)$, and (\romannum2) yields that 
$$
  c_0^{-(K-1)}E_{\mu_K(a,R)} \big[\tau_K^{-1}g\big] \le E_{w,K} [g] \le c_0^{K-1} E_{\mu_K(a,R)} \big[\tau_K^{-1}g\big], 
$$
where $c_0 = c_+/c_-$. 
For bounded and smooth function $f$, we can apply the estimate above to $g = (f - E_{\mu_K(a,R)} [\tau_K^{-1}f])^2$ to obtain 
$$
  E_{w,K} \left[(f - E_{w,K} [f])^2\right] \le E_{w,K} [g] \le c_0^{K-1}E_{\mu_K(a,R)} [\tau_K^{-1}g]. 
$$
On the other hand, take $h_{i,j} = (f - E_{w,K} [f|\sF_{i,j}])^2$ and similarly, 
$$
  E_{\mu_K(a,R)} \left[(\tau_K^{-1}f - E_{\mu_K(a,R)} [\tau_K^{-1}f | \sF_{i,j}])^2\right] \le E_{\mu_K(a,R)} \big[\tau_K^{-1}h_{i,j}\big] \le c_0^{K-1}E_{w,K} [h_{i,j}]. 
$$
Substituting $\tau_K^{-1}f$ for $f$ in Lemma \ref{lem:kac walk}, we get 
$$
  \begin{aligned}
    E_{w,K} \left[(f - E_{w,K} [f])^2\right] &\le c_0^{K-1}E_{\mu_K(a,R)} \left[(\tau_K^{-1}f - E_{\mu_K(a,R)} [\tau_K^{-1}f])^2\right] \\
    &\le \frac{Cc_0^{K-1}}{K} \sum_{i < j} E_{\mu_K(a,R)} \left[(\tau_K^{-1}f - E_{\mu_K(a,R)} [\tau_K^{-1}f | \sF_{i,j}])^2\right] \\
    &\le \frac{Cc_0^{2(K-1)}}{K} \sum_{i < j} E_{w,K} \left[(f - E_{w,K} [f | \sF_{i,j}])^2\right]. 
  \end{aligned}
$$
Since $c_0 = (\delta_+/\delta_-)^{3/2}$, Lemma \ref{lem:weak mean field} then follows. 

Now fix $K \ge 3$ and we construct the map $\tau_K$. 
Write $x_k = (p_k, r_k)$ and define 
$$
\alpha_k = \frac{1}{k}\sum_{i=1}^k r_i, \quad \forall 1 \le k \le K. 
$$
Consider two maps $\zeta$, $\zeta_*: \bR^K \to \bR^K$. 
The first map $\zeta$ is given by 
$$
\zeta(r_1, \ldots, r_K) = (r'_1, \ldots, r'_K), 
$$
such that $r'_K = \alpha_K$, and for $1 \le k \le K - 1$, 
$$
(r'_k)^2 = \frac{2k}{k + 1}\big(V(r_{k+1}) + kV(\alpha_k) - (k + 1)V(\alpha_{k+1})\big), 
$$
where the sign of $r'_k$ is chosen in accordance with $r_k - \alpha_K$. 
Meanwhile, $\zeta_*$ is given by 
$$
\zeta_*(r'_1, \ldots, r'_K) = (r''_1, \ldots, r''_K), 
$$
such that 
\begin{equation*}
  r''_k = 
  \begin{cases}
    r'_K - \sum_{i=1}^{K-1} \frac{r'_i}{i}, &\text{for } k = 1, \\
    r'_K + r'_{k-1} - \sum_{i=k}^{K-1} \frac{r'_i}{i}, &\text{for } 2 \le k \le K-1, \\
    r'_K + r'_{K-1}, &\text{for } k = K. 
  \end{cases}
\end{equation*}
Denote by $J$ and $J_*$ the Jacobian matrices of $\zeta$ and $\zeta_*$, respectively. 
To compute $J$, noticing that $\partial_{r_i}r'_k = \partial_{r_k}r'_k$ for all $i \le k$, and $\partial_{r_i}r'_k = 0$ for all $i > k + 1$, we have 
\begin{equation*}
  J = 
  \begin{bmatrix}
    \frac{\partial r'_1}{\partial r_1} & \frac{\partial r'_1}{\partial r_2} & 0 & \hdots & 0 \\
    \frac{\partial r'_2}{\partial r_2} & \frac{\partial r'_2}{\partial r_2} & \frac{\partial r'_2}{\partial r_3} & \hdots & 0 \\
    \vdots & \vdots & \vdots & & \vdots \\
    \frac{\partial r'_{K-1}}{\partial r_{K-1}} & \frac{\partial r'_{K-1}}{\partial r_{K-1}} & \frac{\partial r'_{K-1}}{\partial r_{K-1}} & \hdots & \frac{\partial r'_{K-1}}{\partial r_K} \\
    \frac{\partial r'_K}{\partial r_K} & \frac{\partial r'_K}{\partial r_K} & \frac{\partial r'_K}{\partial r_K} & \hdots & \frac{\partial r'_K}{\partial r_K} & 
  \end{bmatrix}. 
\end{equation*}
Hence, its determinant reads 
$$
|\det(J)| = \left|\frac{\partial r'_K}{\partial r_K}\right| \cdot \prod_{k=1}^{K-1} \left|\frac{\partial r'_k}{\partial r_k} - \frac{\partial r'_k}{\partial r_{k+1}}\right|. 
$$
Since $\partial_{r_K}r'_K = 1/K$ and for $k = 1, \ldots, K - 1$ we have 
$$
  \frac{\partial r'_k}{\partial r_i} = 
  \begin{cases}
    \frac{k}{(k + 1)r'_k}[V'(\alpha_k) -V'(\alpha_{k+1})], &\text{if } 1 \le i \le k, \\
    \frac{k}{(k + 1)r'_k}[V'(r_{k+1}) -V'(\alpha_{k+1})], &\text{if } i = k + 1. 
  \end{cases}
$$
In consequence, $|\det(J)|$ equals to 
$$
  \frac{1}{K}\prod_{k=1}^{K-1} \frac{\sqrt{k}}{\sqrt{2(k+1)}}\frac{|V'(r_{k+1}) - V'(\alpha_k)|}{\sqrt{V(r_{k+1}) + kV(\alpha_k) - (k+1)V(\alpha_{k+1})}}. 
$$
Applying the estimate in \eqref{eq:estimate on jacobian-k} to obtain that 
$$
\frac{c_-^{K-1}}{K} \le |\det(J)| \le \frac{c_+^{K-1}}{K}. 
$$
Meanwhile it is easy to calculate that $|\det(J_*)| = K$. 
Therefore, define 
$$
\tau_K: (p_1, \ldots, p_K, r_1, \ldots, r_K) \mapsto (p_1, \ldots, p_K, r''_1, \ldots, r''_K), 
$$
then $|\det(\tau'_K)|$ satisfies (\romannum2). 
On the other hand, suppose that $\{x_k = (p_k, r_k)\}_{1\le k\le K}$ belongs to the microcanonical manifold $\Omega_{w,K}$, then $r'_K = r$ and 
$$
\frac1K\sum_{k=1}^{K-1} \frac{k + 1}{k}\frac{(r'_k)^2}{2} = \frac1K\sum_{k=1}^{K} V(r_k) - V(r'_K) = e - \frac1K\sum_{k=1}^K \frac{p_k^2}2 - V(r). 
$$
Then, it follows from the definition of $r''_k$ that 
$$
\frac1K\sum_{k=1}^{K} (r''_k)^2 = (r'_K)^2 + \frac1K\sum_{k=1}^{K-1} \frac{k+1}{k}(r'_k)^2 = 2e - 2V(r) + r^2 - \frac1K\sum_{k=1}^K p_k^2. 
$$
Hence, $\tau_K(x_1, \ldots, x_K) \in S_K(a, R)$ with $R = 2e - V(r) + r^2$, and (\romannum1) is also verified. 
The proof of Lemma \ref{lem:weak mean field} is then completed. 
\end{proof}

\section{Tightness}
\label{sec:tightness}

In Section \ref{sec:equilibrium fluctuation} we have proved the convergence of the finite-dimensional distribution of $\{\bQ_N\}$. 
In order to complete the proof of Theorem \ref{thm:hyperbolic}, we need its tightness in $C([0, T], \sH_{-k}(\lambda))$. 
The proof is standard, and we summarize it here. 

It suffices to show the two statements below: 
\begin{equation}
 \label{eq:tightness1}
  \lim_{M\to\infty} \limsup_{N\to\infty} \bP_{\lambda,N} \left\{\sup_{t\in[0,T]} \|Y_N(t)\|_{-k} \ge M\right\} = 0, 
\end{equation}
\begin{equation}
  \label{eq:tightness2}
  \lim_{\delta\downarrow0} \limsup_{N\to\infty} \bP_{\lambda,N} \big\{w_{-k}(Y_N, \delta) \ge \epsilon\big\} = 0, \quad \forall \epsilon > 0, 
\end{equation}
where $w_{-k}(Y_N, \delta)$ is the modulus of continuity in $C([0, T], \sH_{-k}(\lambda))$. 
Recall that 
$$
  \|Y_N\|_{-k}^2 = \sum_{i=1}^2 \sum_{n=0}^\infty \left\{\theta_n^{-2k}Y_N^2\big(R\bsm_{i,n}\big) + \kappa_n^{-2k}Y_N^2\big(R\bsn_{i,n}\big)\right\}, 
$$
where $R$ is the rotation matrix in \eqref{eq:rotation}, and $\bsm_{i,n}$, $\bsn_{i,n}$ are the three-dimensional Fourier bases defined in \eqref{eq:fourier basis bdc} and \eqref{eq:fourier basis no bdc}. 

Take $f = \bsm_{i,n}$ or $\bsn_{i,n}$ for some $(i, n)$. 
Applying \eqref{eq:ito} with $H(t) \equiv Rf$, 
$$
  Y_N(t, Rf) = Y_0(0, Rf) + \int_0^t Y_N(s, L^*[Rf])ds + \epsilon_N(t, f), 
$$
and by Lemma \ref{lem:i2 and m} and Remark \ref{rem:i1 weak}, $\epsilon_N$ satisfies that 
$$
  \bE_{\lambda,N} \left[\sup_{t\in[0,T]} \epsilon_N^2(t, f)\right] = o_N(1)\big(|f'|_\infty^2 + |f''|_\infty^2 + \|f'\|^2\big). 
$$
On the other hand, it is easy to see that 
$$
  \bE_{\lambda,N} \left[\sup_{t\in[0,T]} \left|\int_0^t Y(s, L^*[Rf])ds\right|^2\right] \le CT^2\|f'\|^2. 
$$
Observe that $|f''_{i,n}(x)| \le \sqrt2\max\{\theta_n^2, \kappa_n^2\}$. 
Then, for $k > 5/2$, \eqref{eq:tightness1} and \eqref{eq:tightness2} can be proved by standard arguments (cf. \cite[11.3]{KL99}). 

\section{Equivalence of ensembles}
\label{sec:equivalence of ensembles}

In this section we prove the equivalence of ensembles for the dynamics with multi-dimensional conserved quantities. 
By applying Proposition \ref{prop:equivalence of ensembles} to the the model in this paper, we obtain Corollary \ref{cor:equivalence of ensembles}, which is necessary in the proof of Lemma \ref{lem:i1}. 

The notations in this section are different from the former part. 
Let $\pi$ be a Borel measure on $\Omega = \bR^m$ with smooth density function with respect to the Lebesgue measure, and $\bff = (f_1, \ldots, f_d)$ be a $d$-dimensional function on $\Omega$ with compact level sets. 
Suppose that there is some domain $D \subseteq \bR^d$, such that 
$$
  Z(\lambda) \triangleq \log \left[\int_\Omega \exp \big\{\lambda \cdot \bff(\omega)\big\}\pi(d\omega)\right] < \infty, \quad \forall \lambda \in D. 
$$
To avoid the problem of regularity, we assume that $Z$ is four times continuously differentiable on $D$, and its Hessian matrix $\Sigma_\lambda = Z''(\lambda)$ is always positive-definite. 
To simplify the notations we denote $\bfu_\lambda = \nabla_\lambda Z(\lambda)$. 

For $\lambda \in D$ we can define the \emph{tilted probability measure} by 
$$
  \pi_\lambda(d\omega) \triangleq \exp\{\lambda\cdot\bff(\omega) - Z(\lambda)\}\pi(d\omega). 
$$
Observe that $E_{\pi_\lambda} [\bff] = \bfu_\lambda$, and $E_{\pi_\lambda} [(\bff - \bfu_\lambda)(\bff - \bfu_\lambda)'] = \Sigma_\lambda$. 
Let $\Phi_\lambda$ be the centered characteristic function of $\bff$ with respect to $\pi_\lambda$, given by 
$$
  \Phi_\lambda(\bfh) = \int_\Omega \exp \big\{i\bfh \cdot (\bff(\omega) - \bfu_\lambda)\big\}\pi_\lambda(d\omega), \quad \forall \bfh \in \bR^d. 
$$
We also assume that there exists some $\epsilon_0 > 0$, such that 
$$
  \sup_{\bfh\in\bR^d} |\bfh|^{\epsilon_0}|\Phi_\lambda(\bfh)| < \infty. 
$$

The main methods we use here is a multi--dimensional local central limit theorem with an edge expansion and a large deviation property for $\bff$. 
We state them in Lemma \ref{lem:local clt} and Lemma \ref{lem:large deviation} respectively. 
Let $\phi_\lambda = \phi_\lambda(\bfx)$ be the Gaussian density function on $\bR^d$, whose mean is $\mathbf 0$ and variance matrix is $\Sigma_\lambda$: 
$$
  \phi_\lambda(\bfx) = \frac1{(2\pi)^{d/2}}\frac1{\sqrt{\det\Sigma_\lambda}}\exp\left\{-\frac{\bfx'\Sigma_\lambda^{-1}\bfx}2\right\}, \quad \forall \bfx \in \bR^d. 
$$
For $k \in \bN$, define the $d$-varibale polynomials $P_{\lambda,k}$ by 
$$
  P_{\lambda,k}(\bfh) = \sum_{|\alpha|=k} \frac{\partial_\alpha Z(\lambda)}{\alpha!}\bfh^\alpha, 
$$
where $\alpha = (\alpha_1, \ldots, \alpha_d)$ is multiple index, $\alpha_j \ge 0$, and 
$$
  |\alpha| = \sum_{j=1}^d \alpha_j, \quad \alpha! = \prod_{j=1}^d \alpha_j!, \quad \partial_\alpha = \prod_{j=1}^d \frac{\partial^{\alpha_j}}{\partial \lambda_j^{\alpha_j}}, \quad \bfh^\alpha = \prod_{j=1}^d h_j^{\alpha_j}. 
$$
Also define the polynomials $Q_{\lambda,3}$ and $Q_{\lambda,4}$ by 
$$
  Q_{\lambda,3} = \frac1{(2\pi)^d\phi_\lambda}\int_{\bR^d} \exp\left\{-i\bfx \cdot \bfh - \frac{\bfh'\Sigma_\lambda\bfh}2\right\}P_{\lambda,3}(i\bfh)d\bfh; 
$$
$$
  Q_{\lambda,4} = \frac1{(2\pi)^d\phi_\lambda}\int_{\bR^d} \exp\left\{-i\bfx \cdot \bfh-\frac{\bfh'\Sigma_\lambda\bfh}2\right\}\left(P_{\lambda,4} + \frac{P_{\lambda,3}^2}2\right)(i\bfh)d\bfh. 
$$
Let $\Omega_n$ be the $n$-product space of $\Omega$. 
Define 
$$
  \bff_{(n)}(\vec\omega) = \frac1{n}\sum_{j=1}^n \bff(\omega_j), \quad \forall \vec\omega = (\omega_1, \ldots, \omega_n) \in \Omega_n. 
$$
Equip $\Omega_n$ with the product measure $\pi_{\lambda,n} = \otimes_j \pi_\lambda(d\omega_j)$. 
We have the following local central limit theorem. 
The proof is standard \cite[Theorem \Romannum{7}.15]{Petrov75}. 

\begin{lem}
\label{lem:local clt}
Let $f_{\lambda,n}$ be the density function of the random vector 
$$
  \sqrt n\big(\bff_{(n)} - \bfu_\lambda\big) = \frac1{\sqrt n}\sum_{j=1}^n \big(\bff(\omega_j) - \bfu_\lambda\big) 
$$
with respect to the product measure $\pi_{\lambda,n}$ for $n$ large enough. 
Then, 
\begin{equation}
\label{eq:local clt}
  \left|f_{\lambda,n}(\bfx) - \phi_\lambda(\bfx)\left(1 + \frac{Q_{\lambda,3}(\bfx)}{\sqrt n} + \frac{Q_{\lambda,4}(\bfx)}{n}\right)\right| \leq \frac{K_{\lambda,n}}n, \quad \forall \bfx \in \bR^d, 
\end{equation}
where $\lim_{n\to\infty} K_{\lambda,n} = 0$, uniformly in any compact subset of $D$. 
\end{lem}

As $Z$ is strictly convex, consider its \emph{Fenchel-Legendre transform}: 
$$
  Z^*(\bfu) = \sup_{\lambda \in D}\{\lambda \cdot \bfu - Z(\lambda)\}. 
$$
Let $D^* = \{\bfu \in \bR^d: Z^*(\bfu) < \infty\}$. 
The superior is reached at a unique $\lambda(\bfu) \in D$, given by the convex conjugate 
$$
  \lambda(\bfu) = \nabla_\bfu Z^*(\bfu), \quad \bfu_\lambda = \nabla_\lambda Z(\lambda). 
$$
Notice that $\bfu \mapsto \lambda(\bfu)$ and $\lambda \mapsto \bfu_\lambda$ are a pair of inverse maps between $D$ and $D^*$. 
For $\lambda \in D$ and $\bfu \in D^*$, define the rate function $I_\lambda(\bfu)$ by 
\begin{equation}
\label{eq:rate function}
  I_\lambda(\bfu) = Z^*(\bfu) - Z^*(\bfu_\lambda) - \nabla_\bfu Z^*(\bfu_\lambda) \cdot (\bfu - \bfu_\lambda). 
\end{equation}
Denote by $M_\lambda$ the largest eigenvalue of $\Sigma_\lambda$. 
By the arguments above it is not hard to conclude that for any constant $M > M_\lambda$, we have 
\begin{equation}
\label{eq:estimate on rate function}
  I_\lambda(\bfu) \ge (2M)^{-1}|\bfu - \bfu_\lambda|^2 
\end{equation}
holds if $|\bfu - \bfu_\lambda|$ is small enough. 
By virtue of \eqref{eq:estimate on rate function}, we can also obtain the following large deviation property for $\bff_{(n)}$. 

\begin{lem}
\label{lem:large deviation}
For any $M > M_\lambda$, there exists some $\delta_M$ such that 
$$
  \pi_{\lambda,n} \left\{|\bff_{(n)} - \bfu_\lambda| \ge \delta\right\} \le 2^d\exp\left(-\frac{nM\delta^2}{d}\right), 
$$
holds for all $n \ge 1$ when $|\delta| < \delta_M$. 
\end{lem}

\begin{proof}
Let $\Gamma \subseteq \bR^d$ be the collection of vectors whose coordinates are all $\pm 1$. 
Notice that the following inequality holds for all $\bfx \in \bR^d$: 
$$
  e^{|\bfx|} \le \prod_{j=1}^d e^{|x_j|} \le \prod_{j=1}^d (e^{-x_j} + e^{x_j}) = \sum_{\gamma\in\Gamma} e^{\gamma\cdot\bfx}. 
$$
By exponential Chebyshev's inequality and the above estimate, for $\theta > 0$, 
\begin{equation*}
  \begin{aligned}
    \pi_{\lambda,n} \left\{|\bff_{(n)} - \bfu_\lambda| \ge \delta\right\} &\le \sum_{\gamma\in\Gamma} e^{-n\theta\delta}\int_{|\bff_{(n)} - \bfu_\lambda| \ge \delta} \exp\big\{n\theta \gamma \cdot (\bff_{(n)}-\bfu_\lambda)\big\}\pi_{\lambda,n}(d\vec\omega) \\
    &\le \sum_{\gamma\in\Gamma} \exp\big\{\!-n\theta u' + nZ(\lambda + \theta\gamma) - nZ(\lambda)\big\}, 
  \end{aligned}
\end{equation*}
where $u' = \gamma \cdot \bfu_\lambda + \delta$. 
To optimize this estimate, define 
$$
  I_{\lambda,\gamma}(u') = \sup_{\theta > 0} \{\theta u' - Z(\lambda + \theta\gamma) + Z(\lambda)\} = \sup_{\theta \in \bR} \{\theta u' - Z(\lambda + \theta\gamma) + Z(\lambda)\}. 
$$
The last equality is due to the fact that $u' - \partial_\theta Z(\lambda + \theta\gamma)|_{\theta = 0} = \delta > 0$. 
Notice that $I_{\lambda,\gamma}$ is the rate function defined in \eqref{eq:rate function} corresponding to the measure $\pi_\lambda$ and the function $\gamma \cdot \bff$. 
By the arguments which has been used to derive \eqref{eq:estimate on rate function}, one obtains that $I_{\lambda,\gamma}(u') \ge M_\lambda|\gamma|^{-2}\delta^2$. 
The estimate in Lemma \ref{lem:large deviation} then follows directly. 
\end{proof}

Now fix some $k \in \bN$. 
For an integrable function $G$ on $\Omega_k$, any $n \ge k$ and $\bfu \in D^*$, define the microcanonical expectation $\langle G|\bfu \rangle_n$ by 
$$
  \langle G|\bfu \rangle_n = E_{\pi_{\lambda,n}} [G~|~\bff_{(n)} = \bfu]. 
$$
It is easy to see that the definition of $\langle G|\bfu \rangle_n$ does not depend on $\lambda$. 
Notice that though the conditional expectation can usually be defined only in a almost sure sense, under the regularity of $\bff$, the microcanonical surface 
$$
  \Omega_{\bfu,n} = \{\vec\omega \in \Omega_n; \bff_{(n)}(\vec\omega) = \bfu\}, 
$$
is smooth enough to define the regular conditional expectation for everywhere in $D^*$. 
Recall that $\bfu_\lambda = E_{\pi_\lambda} [\bff]$. The following estimate (cf. \cite[p.353, Corollary A2.1.4]{KL99}) holds. 

\begin{prop}
\label{prop:equivalence of ensembles}
Suppose that for some compact subset $D_0$ of $D$, 
$$
  C_j \triangleq \sup_{\lambda \in D_0}E_{\pi_\lambda} \big[|\bff - \bfu_\lambda|^j\big] < \infty, \quad j = 1, 2, 3, 4, 
$$
and $G: \Omega_k \to \bR$ satisfies that $E_{\pi_{\lambda,k}} [G^2] < \infty$ for all $\lambda \in D_0$. 
Then, 
\begin{equation}
\label{eq:equivalence of ensembles}
  \limsup_{n\to\infty} n\big|\langle G|\bfu_\lambda \rangle_n - E_{\pi_{\lambda,k}} [G]\big| \le Ck\sqrt{E_{\pi_{\lambda,k}} \big[(G - E_{\pi_{\lambda,k}} [G])^2\big]}, 
\end{equation}
with a uniform constant $C$ for every $\lambda \in D_0$. 
\end{prop}

\begin{proof}
The proof is exactly parallel to \cite[Corollary A2.1.4]{KL99}. 
We sketch it for completeness. 
Without loss of generality we can assume that $E_{\pi_{\lambda,k}} [G] = 0$ for some fixed $\lambda \in D_0$. 
Denote by $F_{\lambda,n}$ the density function of $\bff_{(n)}$ under $\pi_{\lambda,n}$: 
$$
  \int_{\Omega_n} g(\bff_{(n)})d\pi_{\lambda,n} = \int_{\bR^d} g(\bfu)F_{\lambda,n}(\bfu)d\bfu 
$$
for all integrable function $g$ on $\bR^d$. 
We can write $\langle G|\bfu \rangle_n$ as 
$$
  \int_{\bR^k} G(\omega_1, \ldots, \omega_k)\left(\frac{F_{\lambda,n-k}(\bfu_{k,n})}{F_{\lambda,n}(\bfu)} - 1\right)\pi_{\lambda,k}(d\omega_1\ldots d\omega_k), 
$$
where $\bfu_{k,n} = (n - k)^{-1}(n\bfu - k\bff_{(k)})$. 
Schwarz inequality then yields that 
$$
  \langle G|\bfu \rangle_n^2 \le E_{\pi_{\lambda,k}} \big[G^2\big]E_{\pi_{\lambda,k}} \left[\left|\frac{F_{\lambda,n-k}(\bfu_{k,n})}{F_{\lambda,n}(\bfu)} - 1\right|^2\right]. 
$$
Take $\bfu = \bfu_\lambda$ in the above expression. 
By Lemma \ref{lem:local clt}, 
$$
  \left|\frac{F_{\lambda,n-k}(\bfu_{k,n})}{F_{\lambda,n}(\bfu_\lambda)} -1\right| \le \frac{Ck}n\big(1 + |\bff_{(k)} - \bfu_\lambda| + k|\bff_{(k)} - \bfu_\lambda|^2\big). 
$$
where $C$ is a constant depending on $\{C_j; j = 1, 2, 3, 4\}$, the polynomials $Q_{\lambda,3}$, $Q_{\lambda,4}$ and the sequence $K_{\lambda,n}$ appeared in \eqref{eq:local clt}. 
Hence, \eqref{eq:equivalence of ensembles} holds for the fixed $\lambda$ we chosen. 
Since the polynomials $Q_{\lambda,3}$ and $Q_{\lambda,4}$ are continuously dependent on $\lambda$, and $K_{\lambda,n}$ vanishes uniformly in $D_0$, we can extend the result to every $\lambda \in D_0$. 
\end{proof}

Now we apply Proposition \ref{prop:equivalence of ensembles} to the model established in Section \ref{sec:introduction}. 
Let $\Omega = \bR^2$, $\pi$ be the Lebesgue measure, and $\bff$ be the three-dimensional function on $\Omega$ given by 
$$
  \bff(\omega) = (p, r, - p^2/2 - V(r)), \quad \text{for}\ \omega = (p, r) \in \Omega, 
$$
where $V$ is a $C^4$-smooth function with quadratic growth \eqref{eq:assumptions}. 
It is not hard to obtain that $D = \bR^2\times\bR_+$ and $D^* = \bR^2\times\bR_-$. 
For $\lambda \in D$, 
$$
  Z(\lambda) = \ln\left(\int_\bR e^{-\lambda_3V(r) + \lambda_2r}dr\right) + \frac{\lambda_1^2}{2\lambda_3} + \frac12\ln\left(\frac{2\pi}{\lambda_3}\right), \quad \lambda = (\lambda_1, \lambda_2, \lambda_3). 
$$
So $Z$ is four times differentiable and all of its partial derivatives are uniformly bounded in $[-K, K]^2\times[\epsilon, \infty)$ for $K$, $\epsilon > 0$. 
Furthermore, the assumptions in Proposition \ref{prop:equivalence of ensembles} holds in the same set. 
Recall the continuous map $\bfu \to \lambda(\bfu)$ form $D^*$ and $D$, which gives the inverse of $\lambda \to \bfu_\lambda$. 
With Proposition \ref{prop:equivalence of ensembles}, we have the following estimate. 

\begin{cor}
\label{cor:equivalence of ensembles}
Suppose that $F$ is a function on $\Omega_k$, such that $E_{\pi_{\lambda,k}} [F]$ is twice continuously differentiable in $\lambda$, and $E_{\pi_{\lambda',k}} [F^4] < \infty$ for some fixed $\lambda' \in D$. 
Define 
$$
  G = F - E_{\pi_{\lambda',k}} [F] - \nabla_\bfu E_{\pi_{\lambda(\bfu),k}}|_{\bfu = \bfu'} \cdot (\bff(\omega_j) - \bfu'), 
$$
where $\bfu' = \bfu_{\lambda'} \in D^*$. 
Then for $n$ large enough, we have 
$$
  E_{\lambda',n} \big[\langle G|\bfu \rangle_n^2\big] \le Cn^{-2}, 
$$
where $C$ is a finite constant depending only on $F$ and $\lambda'$. 
\end{cor}

\begin{proof}
Fix some $\delta \in (0, \delta_M)$, where $\delta_M$ is the constant appeared in Lemma \ref{lem:large deviation}. 
By Schwarz inequality and Lemma \ref{lem:large deviation}, 
$$
  E_{\lambda',n} \big[\langle G|\bfu \rangle_n^2\mathbf1_{\{|\bfu-\bfu'|>\delta\}}\big] \le 2^{\frac d2}\exp\left\{-\frac{nM\delta^2}{2d}\right\}\sqrt{E_{\lambda',n} [\langle G|\bfu \rangle_n^4]}, 
$$
so it suffices only consider the compact set $\{|\bfu - \bfu'| \le \delta\}$. 
Observe that 
$$
  \langle G|\bfu \rangle_n = \langle F|\bfu \rangle_n - E_{\pi_{\lambda',k}} [F] - \nabla_\bfu E_{\pi_{\lambda(\bfu),k}}|_{\bfu = \bfu'} \cdot (\bfu - \bfu'), \quad \forall \bfu \in D^*. 
$$
Recall that $\lambda(\bfu)$ continuous in $D^*$, so $\{\lambda(\bfu); |\bfu - \bfu'| \le \delta\}$ is a compact subset of $D$. 
Apply Proposition \ref{prop:equivalence of ensembles} with $\lambda = \lambda(\bfu)$ to obtain that 
$$
  \big|\langle F|\bfu \rangle_n - E_{\pi_{\lambda(\bfu),k}} [F]\big| \le Cn^{-1}, \quad \forall \bfu \in \{|\bfu - \bfu'| \le \delta\}, 
$$
where the constant $C = C(F, \delta)$, so its square integral is bounded by $C'n^{-2}$. 
We are left with the second moment in $\{|\bfu - \bfu'| \le \delta\}$ of 
$$
  E_{\pi_{\lambda(\bfu),k}} [F] - E_{\pi_{\lambda',k}} [F] - \nabla_\bfu E_{\pi_{\lambda(\bfu),k}}|_{\bfu = \bfu'} \cdot (\bfu - \bfu'). 
$$
Since $E_{\pi_{\lambda,n}} [F]$ is smooth in $\lambda$ and $\lambda(\bfu)$ is smooth in $\bfu$, we know that this function is bounded by $C|\bfu - \bfu_\lambda|$, with some constant $C = C(F, \lambda')$. 
The desired estimate then follows from the fact that $E_{\lambda',n} [|\bff_{(n)} - \bfu'|^4] \le C'n^{-2}$. 
\end{proof}

\section*{Acknowledgments} 
This work has been partially supported by the grants ANR-15-CE40-0020-01 LSD 
of the French National Research Agency. 
We thank Makiko Sasada for the insightful discussion about the spectral gap estimate (cf. Section \ref{sec:spectral gap}). 

\bibliography{[BIB]common.bib,[BIB]chain_of_oscillators.bib}

\noindent
{Stefano Olla\\
CEREMADE, UMR-CNRS, Universit\'e de Paris Dauphine, PSL Research University}\\
{\footnotesize Place du Mar\'echal De Lattre De Tassigny, 75016 Paris, France}\\
{\footnotesize \tt olla@ceremade.dauphine.fr}\\
\\
{Lu Xu\\
CEREMADE, UMR-CNRS, Universit\'e de Paris Dauphine, PSL Research University}\\
{\footnotesize Place du Mar\'echal De Lattre De Tassigny, 75016 Paris, France}\\
{\footnotesize \tt xu@ceremade.dauphine.fr}\\

\end{document}